\documentclass[reqno]{amsart}
\usepackage{CJK}
\usepackage{hyperref}
\usepackage{cite}
\usepackage{amsmath}
\usepackage{mathrsfs}
\usepackage{xcolor}
\usepackage{txfonts}
\usepackage{bbm}
\usepackage{pifont}
\usepackage{bbding}
\usepackage{amsmath, esint}
\usepackage{mathrsfs}
\usepackage{amsfonts}
\usepackage{amssymb}
\usepackage{amsfonts,amssymb,amsmath,indentfirst,cases,amsthm}
\usepackage{epsfig}
\usepackage[numbers,sort&compress]{natbib}

\numberwithin{equation}{section}

\newtheorem{theorem}{Theorem}[section]
\newtheorem{lemma}[theorem]{Lemma}
\newtheorem{definition}[theorem]{Definition}

\newtheorem{proposition}[theorem]{Proposition}

\newtheorem{corollary}[theorem]{Corollary}

\allowdisplaybreaks

\begin{document}
	
\title[\hfil Nonlocal double phase equations in the Heisenberg Group]{Local regularity for nonlocal double phase equations in the Heisenberg group}

\author[Y. Fang, C. Zhang and J. Zhang  \hfil \hfilneg]{Yuzhou Fang, Chao Zhang and Junli Zhang$^*$}

\thanks{$^*$Corresponding author.}

\address{Yuzhou Fang \hfill\break School of Mathematics, Harbin Institute of Technology, Harbin 150001, China}
\email{18b912036@hit.edu.cn}

\address{Chao Zhang  \hfill\break School of Mathematics and Institute for Advanced Study in Mathematics, Harbin Institute of Technology, Harbin 150001, China}
\email{czhangmath@hit.edu.cn}

\address{Junli Zhang  \hfill\break School of Mathematics and Data Science, Shaanxi University of Science and Technology, Xi'an 710021, China}
\email{jlzhang2020@163.com}


\maketitle

\begin{abstract}
We prove interior boundedness and H\"{o}lder continuity for the weak solutions of nonlocal double phase equations in the Heisenberg group $\mathbb{H}^n$. This solves a problem raised by Palatucci and Piccinini et. al. in 2022 and 2023 for nonlinear integro-differential problems in the Heisenberg group $\mathbb{H}^n$. Our proof of the a priori estiamtes bases on the spirit of De Giorgi-Nash-Moser theory, where the important ingredients are  Caccioppoli-type inequality and Logarithmic estimate.  To achieve this goal, we establish a new and crucial Sobolev-Poincar\'{e} type inequality in local domain, which may be  of independent interest and potential applications.
\end{abstract}

\section{Introduction}
\label{sec-1}

In this paper, we are interested in local behaviour of the weak solutions to nonlocal double phase problem in the Heisenberg group ${{\mathbb{H}}^n}$, whose prototype is
\begin{equation}\label{eqy11}
\mathrm{P.V.}\int_{\mathbb{H}^n}\left[\frac{|u(\xi)-u(\eta)|^{p-2}(u(\xi)-u(\eta))}{\|\eta^{-1}\circ\xi\|^{Q+sp}_{\mathbb{H}^n}}+a(\xi,\eta)
\frac{|u(\xi)-u(\eta)|^{q-2}(u(\xi)-u(\eta))}{\|\eta^{-1}\circ\xi\|^{Q+tq}_{\mathbb{H}^n}}\right]\,d\eta=0 \quad \text{in } \Omega,
\end{equation}
where $1<p\le q<\infty$, $s,t\in(0,1)$, $ a(\cdot,\cdot)\ge 0$, $Q=2n+2$ is the homogeneous dimension and $\Omega$ is an open bounded subset of $\mathbb{H}^n$ ($n\ge1$). The leading operator could change between two different fractional elliptic phase according to whether the modulating coefficient $a$ is zero or not, which is the significant characteristic of this integro-differential equation. In the display above, $\|\cdot\|_{\mathbb{H}^n}$ and $\mathrm{P.V.}$ mean the standard Heisenberg norm and ``in the principal value sense", respectively.

We observe that, if the coefficient $a\equiv0$, Eq. \eqref{eqy11} is reduced to the $p$-fractional subLaplace equation arising in many diverse contexts, such as quantum mechanics, image segmentation models, ferromagnetic analysis and so on. Let us firstly pay attention to the linear scenario, i.e., $p=2$. This kind of problems can be regarded as an extension of the conformally invariant fractional subLaplacian $\left(-\Delta_{\mathbb{H}^n}\right)^s$ in $\mathbb{H}^n$ proposed initially in \cite{BFM13} by the spectral formula
$$
\left(-\Delta_{\mathbb{H}^n}\right)^s:= {2^s}{\left| T \right|^s}\frac{{\Gamma \left( { - \frac{1}{2}{\Delta _H}{{\left| T \right|}^{ - 1}} + \frac{{1 + s}}{2}} \right)}}{{\Gamma\left( { - \frac{1}{2}{\Delta _H}{{\left| T \right|}^{ - 1}} + \frac{{1 - s}}{2}} \right)}},\;\;s \in \left( {0,1} \right),
$$
where $s\in(0,1)$, $\Gamma(\cdot)$ is the Euler Gamma function, $T$ is the vertical vector field, and $\Delta_{\mathbb{H}^n}$ is the typical Kohn-Spencer subLaplacian on $\mathbb{H}^n$. Subsequently, Roncal, Thangavelu \cite{RT16} demonstrated the representation as below
\begin{equation}\label{D2}
\left(-\Delta_{\mathbb{H}^n}\right)^su(\xi):=
C(n,s)\mathrm{P.V.}\int_{{\mathbb{H}^n}} {\frac{{u\left( \xi  \right) - u\left( \eta  \right)}}{{\| {{\eta ^{ - 1}} \circ \xi } \|_{{\mathbb{H}^n}}^{Q + 2s}}}d\eta },   \quad \xi\in\mathbb{H}^n
\end{equation}
holds true for $C(n,s)>0$ depending only on $n,s$. During the last decade, several aspects of the fractional operator of the type \eqref{D2} have been investigated extensively, such as  Hardy and uncertainty inequalities on stratified Lie groups \cite{CCR15}, Sobolev and Morrey-type embedding theory for fractional Sobolev space $H^s(\mathbb{H}^n)$ \cite{AM18}, Harnack and H\"{o}lder estimates in Carnot groups \cite{FF15}, Liouville-type theorem \cite{CT16}. One can refer to \cite{FMPPS18,FGMT15,GT22,GT21} and references therein for more results on the linear case. Regarding the nonlinear analogue to \eqref{D2}, that is, the $p$-growth scenario is considered this time. For what concerns the regularity properties of weak solutions to the fractional $p$-subLaplace equations on the Heisenberg group, Manfredini, Palatucci,  Piccinini and Polidoro  \cite{MPP23} first established the interior boundedness and H\"{o}lder continuity via employing the De Giorge-Nash-Moser iteration; see also \cite{PP22} for the nonlocal Harnack inequality, where the asymptotic behaviour of fractional linear operator was proved as well. In addition, as for the obstacle problems connected with the nonlocal $p$-subLaplacian, we refer to \cite{P22} in which Piccinini studied systematically solvability, semicontinuity, boundedness and H\"{o}lder regularity up to the boundary for weak solutions. To some extent, these results mentioned above extend to the Heisenberg framwork many important counterparts in the fractional Euclidean setting in \cite{DKP14,DKP16,IMS16,KKP16}. More interesting estimates or fundamental functional inequalities can be found, for instance, in \cite{KS18,KS20}. 

Equation \eqref{eqy11} could be viewed naturally as the nonlocal version of the classical double phase problem of the following type
\begin{equation}\label{D3}
-\mathrm{div}(|\nabla u|^{p-2}\nabla u+a(x)|\nabla u|^{q-2}\nabla u)=0 \quad \text{in } \Omega.
\end{equation}
Within the Euclidean context, the regularity theory on weak solutions to \eqref{D3} or minimizers of corresponding functionals has been extensively developed, beginning with the pioneering papers of Colombo and Mingione \cite{CM15a,CM15b}. Under $a\in L^\infty_{\rm loc}(\Omega)$ and, $p\le q\le \frac{np}{n-p}$ if $p<n$, or $p\le q<\infty$ if $p\ge n$, local boundedness for $u$ was shown; and further under $u\in L^\infty_{\rm loc}(\Omega)$, $a\in C^{0,\alpha}_{\rm loc}(\Omega)$ and $p\le q\le p+\alpha$, H\"{o}lder continuity of  $u$ was obtained as well, see, e.g., \cite{CM15b,CMM15}.

Very recently, the investigation of nonlocal problems with nonstandard growth, especially of those with $(p,q)$-growth condition, has been attracting increasing attention, however only in the fractional Euclidean spaces. In this respect, De Filippis, Palatucci \cite{DP19} first introduced nonlocal double phase equations of the form \eqref{eqy11}  in the Euclidean spaces, and established H\"{o}lder continuity for bounded viscosity solutions. Weak theory on this class of nonlocal equations was rapidly explored in hot pursuit, for example, \cite{SM22} for self-improving inequalities on bounded weak solutions, \cite{FZ23} for H\"{o}lder regularity and relationship between weak and viscosity solutions in the differentiability exponents $s\ge t$, \cite{BOS22} for H\"{o}lder property with weaker assumption on solutions in the case $s<t$, \cite{Gia} for the sharp H\"{o}lder index and the parabolic version. Concerning more regularity and related results for nonlocal problems possessing nonuniform growth, one can see \cite{BKO,FZ,PT23,GKS21,CKW22} and references therein.

In particular, we would like to mention that Palatucci and Piccinini et. al. in a series of papers \cite{MPP23,PP22,PP23} proposed the open problems about the regualrity of solutions to the so-called nonlocal doubel phase equation  in the Heisenberg group $\mathbb{H}^n$.  In this paper,  influenced by the works \cite{DKP16,BOS22} we answer this question and develop the local regularity theory for the weak solutions of such equations in the Heisenberg group $\mathbb{H}^n$, including  the boundedness and H\"{o}lder continuity of solutions.  The main difficulties which are different from the previous ones are mainly two parts. One is that Eq. \eqref{eqy11}  not only possesses the nonlocal feature of the embraced integro-differential operators and the noneuclidean geometrical structure of the Heisenberg group, but also inherits the typical characteristics exhibited by the (local) double phase problems due to the $(p,q)$-growth condition and the presence of the nonnegative variable coefficient $a$. The other one is that the existing Sobolev embedding theorem,  Lemma \ref{Le102},  cannot be applied to our setting directly. To overcome this point, we have to establish a suitable Sobolev-Poincar\'{e} type inequality on balls in the Heisenberg group $\mathbb{H}^n$. It may be of independent interest and potential applications when investigating regularity properties for some other nonlocal equations  in the Heisenberg group. These difficulties  make the current study more challenging than the fractional $p$-subLaplacian case.

Now we are in a position to state our main contributions.  We first collect some notations, definitions as well as assumptions. Let $s,\;t$ and $p,\;q$ satisfy
\begin{equation}\label{eqy18}
1<p\le q<\infty, \quad 0<s\le t<1,
\end{equation}
and the coefficient $a:{\mathbb{H}^n}\times {\mathbb{H}^n} \to {\mathbb{R}^+}$ fulfill
\begin{equation}\label{eqy19}
  0 \le a\left( {\xi ,\eta } \right) = a\left( {\eta ,\xi } \right) \le {\| a \|_{{L^\infty }}},\;\;\;\;\xi ,\eta  \in {\mathbb{H}^n}
\end{equation}
and
\begin{equation}\label{eqy110}
\left| {a\left( {\xi ,\eta } \right) - a\left( {\xi ',\eta '} \right)} \right| \le {\left[ a \right]_\alpha }{\left( {{\| {{{\xi '}^{ - 1}} \circ \xi } \|_{{\mathbb{H}^n}}} + {\| {{{\eta '}^{ - 1}} \circ \eta } \|_{{\mathbb{H}^n}}}} \right)^\alpha }
\end{equation}
for $\left( {\xi ,\eta } \right),\left( {\xi ',\eta '} \right) \in {\mathbb{H}^n} \times {\mathbb{H}^n}$ and $\alpha  \in \left( {0,1} \right]$.

For convenience, we introduce the following notations:
\begin{equation*}\label{eqy21}
 H\left( {\xi ,\eta ,\tau } \right): = \frac{{{\tau ^p}}}{{\| {{\eta ^{ - 1}} \circ \xi } \|_{{\mathbb{H}^n}}^{sp}}} + a\left( {\xi ,\eta } \right)\frac{{{\tau ^q}}}{{\| {{\eta ^{ - 1}} \circ \xi } \|_{{\mathbb{H}^n}}^{tq}}}, \quad \xi ,\eta  \in {\mathbb{H}^n}\;\hbox{and}\;\;\tau  > 0,
\end{equation*}
and
\[{J_l}( {{\tau _1} - {\tau _2}} ) = {| {{\tau _1} - {\tau _2}}|^{l - 2}}( {{\tau _1} - {\tau _2}} )\]
with ${\tau _1},{\tau _2} \in \mathbb{R}$ and $l \in \{p,q\}$, and
\begin{equation*}\label{eqy22}
\rho \left( {u ;\Omega } \right) = \int_\Omega  {\int_\Omega  {H\left( {\xi ,\eta ,|u(\xi) -u(\eta)|} \right)\frac{{d\xi d\eta }}{{\| {{\eta ^{ - 1}} \circ \xi } \|_{{\mathbb{H}^n}}^Q}}} }
\end{equation*}
for every measurable set $\Omega \subset {\mathbb{H}^n}$ and $u: \Omega \to \mathbb{R}$. A function space related to weak solutions to \eqref{eqy11} is defined as
$${\mathcal{A}}\left( \Omega  \right): = \left\{ {u :{\mathbb{H}^n} \to \mathbb{R}: u {|_\Omega } \in {L^p}\left( \Omega  \right)\;\;\hbox{and}\;\; \iint_{\mathcal{C}_\Omega} H(\xi ,\eta ,|u(\xi)-u(\eta)|)\frac{d\xi d\eta }{\|\eta^{-1}\circ\xi\|_{\mathbb{H}^n}^Q} < \infty } \right\},$$
where 
\begin{equation*}\label{eqy14}
 {{\mathcal{C}}_\Omega }:= \left( {{\mathbb{H}^n} \times {\mathbb{H}^n}} \right)\backslash \left( {\left( {{\mathbb{H}^n}\backslash \Omega } \right) \times \left( {{\mathbb{H}^n}\backslash \Omega } \right)} \right).
\end{equation*}
Additionally, in view of the nonlocal nature of this problem, we need define a tail space
$$
L^{q-1}_{sp}(\mathbb{H}^n):=\left\{u\in L_{\rm loc}^{q-1}(\mathbb H^n) : \int_{\mathbb{H}^n}\frac{|u(\xi)|^{q-1}}{(1+\|\xi\|_{\mathbb{H}^n})^{Q+sp}}
\,d\xi<\infty\right\}
$$
and the nonlocal tail
$$
T\left( {u;\xi_0, r} \right): = \int_{{\mathbb{H}^n}\backslash {B_r}} {\left( {\frac{{{{\left| {u \left( \xi  \right)} \right|}^{p - 1}}}}{{\| {\xi _0^{ - 1} \circ \xi } \|_{{\mathbb{H}^n}}^{Q + sp}}} + {{\| a \|}_{{L^\infty }}}\frac{{{{\left| {u \left( \xi  \right)} \right|}^{q - 1}}}}{{\| {\xi _0^{ - 1} \circ \xi } \|_{{\mathbb{H}^n}}^{Q + tq}}}} \right)\;d\xi .} 
$$
We can notice that the quantity $T$ is finite if $u\in L^{q-1}_{sp}(\mathbb{H}^n)$. 

We now give the definition of weak solutions to \eqref{eqy11}.

\begin{definition}[weak solution]\label{De1}
  If $u \in {\mathcal{A}}\left( \Omega  \right)$ satisfies
 \begin{align} \label{eqy111}
    \iint_{\mathcal{C}_\Omega}\Bigg[\frac{J_p(u(\xi)-u(\eta))(\varphi(\xi)-\varphi(\eta))}{\|\eta^{-1}\circ\xi\|_{\mathbb{H}^n}^{Q+sp}}  + a(\xi,\eta)\frac{J_q(u(\xi)-u(\eta))(\varphi(\xi)-\varphi(\eta))}{\|\eta^{-1}\circ\xi\|_{\mathbb{H}^n}^{Q+tq}}\Bigg]\,d\xi d\eta   = 0
 \end{align}
 for every $\varphi  \in {\mathcal{A}}\left( \Omega  \right)$ with $\varphi=0$ a.e. in ${\mathbb{H}^n}\backslash \Omega $, then we call $u$ a weak solution to \eqref{eqy11}.
\end{definition}

Note that $u\in {\mathcal{A}}( \Omega)$ implies $u\in{HW}^{s,p}\left( \Omega  \right)$, i.e., ${\mathcal{A}}\left( \Omega  \right) \subset {HW}^{s,p}\left( \Omega  \right)$. Hence in this work, we only consider the case $sp \le Q$. Otherwise, the complementary scenario $sp>Q$ ensures the local boundedness and H\"{o}lder continuity directly because of the fractional Morrey embedding in the Heisenberg group \cite{AM18}. Our main results are stated as follows. The first one is local boundedness of weak solutions.

\begin{theorem}\label{Th11}
Let the conditions \eqref{eqy18} and \eqref{eqy19} be in force. If
\begin{equation}\label{eqy112}
\left\{ \begin{array}{l}
p \le q \le \frac{{Qp}}{{Q - sp}}:=p_s^*\;\;\;\;\;\;when\;\;sp < Q,\\
p \le q < \infty \;\;\;\;\;\;\;\;\;\;\;\;\;\;\;\;\;\;\;when\;\;sp \ge Q,
\end{array} \right.\;
\end{equation}
then every weak solution $u\in {\mathcal{A}}(\Omega ) \cap L_{sp}^{q - 1}\left({{\mathbb{H}^n}}\right)$ to \eqref{eqy11} is locally bounded in $\Omega$.
\end{theorem}

Additionally, we can conclude H\"{o}lder regularity of weak solutions to \eqref{eqy11} via supposing $a(\cdot,\cdot)$ is H\"{o}lder continuous and the distance between $q$ and $p$ is small. For simplicity, we denote
$$
\mathrm{\mathbf{data}}:=\mathrm{\mathbf{data}}(n,p,q,s,t,\alpha,[a]_\alpha,\|a\|_{L^\infty},\|u\|_{L^\infty(\Omega')})
$$
as the set of basic parameters intervening in the problem, where $\Omega'\subset\subset\Omega$.

\begin{theorem}\label{Th12}
Let the conditions \eqref{eqy18}--\eqref{eqy110} with $tq\le sp+\alpha$ be in force. If weak solution $u\in {\mathcal{A}}(\Omega ) \cap L_{sp}^{q - 1}\left({{\mathbb{H}^n}}\right)$ to \eqref{eqy11} has local boundedness in $\Omega$, then it is locally H\"{o}lder continuous as well, that is, for any subset $\Omega'\subset\subset\Omega$, $u$ belongs to $C^{0,\beta}_{\rm loc}(\Omega')$ with some $\beta\in\left(0,\frac{sp}{q-1}\right)$ depending on $\mathrm{\mathbf{data}}$.
\end{theorem}

Putting these two theorems above, H\"{o}lder continuity is immediately obtained without local boundedness assumption under the intersecting conditions.

This paper is organized as follows. In Section \ref{Section 2}, we introduce the Heisenberg group and function spaces, and then deduce some needful Sobolev embedding theorems. Section \ref{Section 4} is dedicated to proving local boundedness of weak solutions by the Caccioppoli-type estimate. At last, we shall show that the locally bounded weak solutions to \eqref{eqy11} are H\"{o}lder continuous via establishing Logarithmic-type inequality in Section \ref{Section 5}.

\section{Functional setting}
\label{Section 2}

In this section, we introduce the Heisenberg group ${{\mathbb{H}}^n}$ and some function spaces, and establish several important Sobolev embedding results. The Euclidean space ${{\mathbb{R}}^{2n + 1}}\;(n \ge 1)$ with the group multiplication
\[\;\xi \circ \eta = \left( {{x_1} + {y_1},{x_2} + {y_2}, \cdots ,{x_{2n}} + {y_{2n}},\tau + \tau' + \frac{1}{2}\sum\limits_{i = 1}^n {\left( {{x_i}{y_{n + i}} - {x_{n + i}}{y_i}} \right)} } \right),\]
where $\xi = \left( {{x_1},{x_2}, \cdots ,{x_{2n}},\tau} \right),$ $\eta = \left( {{y_1},{y_2}, \cdots ,{y_{2n}},\tau'} \right) \in {{\mathbb{R}}^{2n+1}},$ leads to the Heisenberg group ${{\mathbb{H}}^n}$. The left invariant vector field on ${{\mathbb{H}}^n}$ is of the form
\begin{equation*}\label{eq112}
  {X_i} = {\partial _{{x_i}}} - \frac{{{x_{n + i}}}}{2}{\partial _\tau},\;{X_{n + i}} = {\partial _{{x_{n + i}}}} + \frac{{{x_i}}}{2}{\partial _\tau},\quad 1 \le i \le n
\end{equation*}
and a non-trivial commutator is
\begin{equation*}\label{eq113}
  T = {\partial _\tau} = \left[ {{X_i},{X_{n + i}}} \right] = {X_i}{X_{n + i}} - {X_{n + i}}{X_i},~1 \le i \le n.
\end{equation*}
We call that ${X_1},{X_2}, \cdots ,{X_{2n}}$ are the horizontal vector fields on ${{\mathbb{H}}^n}$ and $T$ the vertical vector field. Denote the horizontal gradient of a smooth function $u$ on ${{\mathbb{H}}^n}$ by
\begin{equation*}\label{eq114}
{\nabla _H}u = \left( {{X_1}u,{X_2}u, \cdots ,{X_{2n}}u} \right).
\end{equation*}

The Haar measure in ${{\mathbb{H}}^n}$ is equivalent to the Lebesgue measure in ${{\mathbb{R}}^{2n+1}}$. We denote the Lebesgue measure of a measurable set $E \subset {{\mathbb{H}}^n}$ by $\left| E \right|$. For $\xi = \left( {{x_1},{x_2}, \cdots ,{x_{2n}},\tau} \right),$ we define its module as
\[{\| \xi \|_{{{\mathbb{H}}^n}}} = {\left( {{{\left( {\sum\limits_{i = 1}^{2n} {{x_i}^2} } \right)}^2} + {\tau^2}} \right)^{\frac{1}{4}}}.\]
The Carnot-Carath\'{e}odary metric between two points $\xi$ and $\eta$ in ${{\mathbb{H}}^n}$ is the shortest length of the horizontal curve joining them, denoted by $d(\xi,\eta)$. The C-C metric is equivalent to the Kor\`{a}nyi metric: i.e. $d\left( {\xi,\eta} \right) \sim {\| {{\xi^{ - 1}}\circ \eta} \|_{{{\mathbb{H}}^n}}}$. The ball
\begin{equation}\label{eq111}
  {B_r }\left( \xi_0 \right) = \left\{ {\xi \in {{\mathbb{H}}^n}:d\left( {\xi,\xi_0} \right) < r } \right\}
\end{equation}
is defined by the C-C metric $d$. When not important or clear from the context, we will omit the center as follows: $B_r:=B_r( \xi_0)$.

Let $1 \le p < \infty ,\;s \in \left( {0,1} \right)$, and $v:{{\mathbb{H}}^n} \to {\mathbb{R}}$ be a measurable function. The Gagliardo semi-norm of $v$ is defined as
\[{\left[ v \right]_{H{W^{s,p}}\left( {{{\mathbb{H}}^n}} \right)}} = {\left( {\int_{{{\mathbb{H}}^n}} {\int_{{{\mathbb{H}}^n}} {\frac{{{{\left| {v\left( \xi  \right) - v\left( \eta  \right)} \right|}^p}}}{{\| {{\eta ^{ - 1}} \circ \xi } \|_{{{\mathbb{H}}^n}}^{Q + sp}}}\,d\xi } d\eta } } \right)^{\frac{1}{p}}},\]
and the fractional Sobolev spaces $H{W^{s,p}}\left( {{{\mathbb{H}}^n}} \right)$ on the Heisenberg group are defined as
\[H{W^{s,p}}\left( {{{\mathbb{H}}^n}} \right) = \left\{ {v \in {L^p}\left( {{{\mathbb{H}}^n}} \right):{{\left[ v \right]}_{H{W^{s,p}}\left( {{{\mathbb{H}}^n}} \right)}} < \infty } \right\},\]
endowed with the natural fractional norm
\[{\| v \|_{H{W^{s,p}}\left( {{{\mathbb{H}}^n}} \right)}} = {\left( {\| v \|_{{L^p}\left( {{{\mathbb{H}}^n}} \right)}^p + \left[ v \right]_{H{W^{s,p}}\left( {{{\mathbb{H}}^n}} \right)}^p} \right)^{\frac{1}{p}}}.\]
For any open set $\Omega  \subset {{\mathbb{H}}^n}$, we can define similarly fractional Sobolev spaces $H{W^{s,p}}\left( \Omega  \right)$ and fractional norm ${\| v \|_{H{W^{s,p}}\left( \Omega  \right)}}$. The space $HW_0^{s,p}\left( \Omega  \right)$ is the closure of $C_0^\infty \left( \Omega  \right)$ in $H{W^{s,p}}\left( \Omega  \right)$. Throughout this paper, we denote a generic positive constant as $c$ or $C$. If necessary, relevant dependencies on parameters will be illustrated by parentheses, i.e., $c=c(n,p)$ means that $c$ depends on $n,p$. Now we recall the fractional Poincar\'{e} type inequality and Sobolev embedding in the Heisenberg group; see \cite[Proposition 2.7]{P22} and \cite[Theorem 2.5]{KS18}, respectively.

\begin{lemma}[Poincar\'{e} type inequality]\label{Le101}
Let $ p\ge1,\;s \in \left( {0,1} \right)$ and $v \in HW^{s,p}( B_r)$. Then we have
\[\int_{{B_r}} {{{\left| {v - {{\left( v \right)}_r}} \right|}^p}\,d\xi }  \le c{r^{sp}}\int_{{B_r}} {\int_{{B_r}} {\frac{{{{\left| {v\left( \xi  \right) - v\left( \eta  \right)} \right|}^p}}}{{\| {{\eta ^{ - 1}} \circ \xi } \|_{{\mathbb{H}^n}}^{Q + sp}}}\,d\xi d\eta } } ,\]
where $c=c(n,p)>0$, ${\left( v \right)_r} = \fint_{{B_r}} {vd\xi } $.
\end{lemma}

\begin{lemma}\label{Le102}
Let $1 < p < \infty ,\;s \in \left( {0,1} \right)$ such that $sp < Q$. Let also $v:{{\mathbb{H}}^n} \to {\mathbb{R}}$ be a measurable compactly supported function. Then there is a positive constant $c = c\left( {n,p,s} \right)$ such that
\[\| v \|_{{L^{{p_s^*}}}\left( {{{\mathbb{H}}^n}} \right)}^p \le c\left[ v \right]_{H{W^{s,p}}\left( {{{\mathbb{H}}^n}} \right)}^p\]
with ${p_s^*} = \frac{{Qp}}{{Q - sp}}$ being a critical Sobolev exponent.
\end{lemma}

Now we also give the following result, a truncation lemma near $\partial\Omega$.

\begin{lemma}\label{Le103}
Let $ p \ge 1 ,\;s \in \left( {0,1} \right)$ and $v \in HW^{s,p}\left( {{B_r}} \right)$. If $\varphi  \in {C^{0,1}}\left( {{B_r}} \right) \cap {L^\infty }\left( {{B_r}} \right)$, then it holds that $\varphi  v \in HW^{s,p}\left( {{B_r}} \right)$ and ${\| {\varphi v} \|_{H{W^{s,p}}\left( {{B_r}} \right)}} \le c{\| v \|_{H{W^{s,p}}\left( {{B_r}} \right)}}$ with $c>0$ depending on $n,p,s,r\;\hbox{and}\;\varphi.$
\end{lemma}

The proof of this lemma is very similar to that of \cite[Lemma 5.3]{DPV12}, so we omit it here. Based on Lemmas \ref{Le101}--\ref{Le103}, we could conclude a Sobolev-Poincar\'{e} inequality on balls in the Heisenberg group, which plays a crucial role in proving regularity of solutions.

\begin{proposition}[Sobolev-Poincar\'{e} type inequality]\label{Pro104}
Let $1 < p < \infty ,\;s \in \left( {0,1} \right)$ fulfill $sp < Q$. Suppose that $v \in H{W^{s,p}}\left( {{B_R(\xi_0)}} \right)$ and ${B_r(\xi_0)} \subset {B_R(\xi_0)}\;(0<r<R)$ are concentric balls. Then there exists a positive constant $c(n,p,s)$ such that
\[{\left( {{\fint_{{B_r}}}{{\left| {v - {{\left( v \right)}_{{r}}}} \right|}^{p_s^*}}\,d\xi } \right)^{\frac{p}{{p_s^*}}}} \le c{D_1}(R,r){\fint_{{B_R}}}\int_{{B_R}} {\frac{{{{\left| {v\left( \xi  \right) - v\left( \eta  \right)} \right|}^p}}}{{\| {{\eta ^{ - 1}} \circ \xi } \|_{{\mathbb{H}^n}}^{Q + sp}}}\,d\xi } d\eta  + c{D_2}(R,r){\fint_{{B_R}}}{\left| v \right|^p}\,d\xi \]
where
\[{D_1}(R,r): = {r^{sp}}{\left( {\frac{R}{r}} \right)^Q}\left[ {1 + {{\left( {\frac{R}{{R - r}}} \right)}^p} + {{\left( {\frac{R}{r}} \right)}^Q}{{\left( {\frac{R}{{R - r}}} \right)}^p} + {{\left( {\frac{R}{{R - r}}} \right)}^{Q + sp}} + {{\left( {\frac{R}{r}} \right)}^Q}{{\left( {\frac{R}{{R - r}}} \right)}^{Q + sp}}} \right]\]
and
\[{D_2}(R,r): = {\left( {\frac{R}{r}} \right)^{Q - sp}}{\left( {\frac{R}{{R - r}}} \right)^{Q + sp}}\left[ {1 + {{\left( {\frac{R}{r}} \right)}^Q}} \right].\]
\end{proposition}

\begin{proof} Take $\varphi \left( \xi  \right) \in C_0^\infty \left( {{B_R}\left( {{\xi _0}} \right)} \right)$ as a cut-off function such that $0 \le \varphi  \le 1,\;\varphi  \equiv 1$ in ${{B_r}\left( {{\xi _0}} \right)}$, ${\rm supp}\, \varphi \subset B_\frac{R+r}{2}( \xi _0)$ and $\left| {{\nabla _H}\varphi } \right| \le \frac{c}{{R - r}}$ in ${{B_R}\left( {{\xi _0}} \right)}$. Then $\varphi  \in HW_0^{s,p}\left( {\mathbb{H}^n} \right)$ by zero extension.

This time, we by virtue of Lemma \ref{Le102} get
\begin{align*}
   {\left( {\int_{{B_r}} {{{\left| {v - {{\left( v \right)}_r}} \right|}^{p_s^*}}\,d\xi } } \right)^{\frac{p}{{p_s^*}}}} \le & {\left( {\int_{{\mathbb{H}^n}} {{{\left| {\left( {v - {{\left( v \right)}_r}} \right)\varphi } \right|}^{p_s^*}}\,d\xi } } \right)^{\frac{p}{{p_s^*}}}} \\
   \le&  c\int_{{\mathbb{H}^n}} {\int_{{\mathbb{H}^n}} {\frac{{{{\left| {\left( {v - {{\left( v \right)}_r}} \right)\varphi \left( \xi  \right) - \left( {v - {{\left( v \right)}_r}} \right)\varphi \left( \eta  \right)} \right|}^p}}}{{\| {{\eta ^{ - 1}} \circ \xi } \|_{{\mathbb{H}^n}}^{Q + sp}}}\,d\xi } d\eta } \\
    \le &c\int_{{\mathbb{H}^n}} {\int_{{\mathbb{H}^n}} {\frac{{{\varphi ^p}\left( \xi  \right){{\left| {\left( {v - {{\left( v \right)}_r}} \right)\left( \xi  \right) - \left( {v - {{\left( v \right)}_r}} \right)\left( \eta  \right)} \right|}^p}}}{{\| {{\eta ^{ - 1}} \circ \xi } \|_{{\mathbb{H}^n}}^{Q + sp}}}\,d\xi } d\eta } \\
    & + c\int_{{\mathbb{H}^n}} {\int_{{\mathbb{H}^n}} {\frac{{{{\left| {\varphi \left( \xi  \right) - \varphi \left( \eta  \right)} \right|}^p}{{\left| {\left( {v - {{\left( v \right)}_r}} \right)\left( \eta  \right)} \right|}^p}}}{{\| {{\eta ^{ - 1}} \circ \xi } \|_{{\mathbb{H}^n}}^{Q + sp}}}\,d\xi } d\eta } \\
    : =& {I_1} + {I_2}.
\end{align*}
We split ${\mathbb{H}^n} \times {\mathbb{H}^n}$ into
\[\left( {{B_R} \times {B_R}} \right) \cup \left( {{\mathbb{H}^n}\backslash {B_R} \times {B_R}} \right) \cup \left( {{B_R} \times {\mathbb{H}^n}\backslash {B_R}} \right) \cup \left( {{\mathbb{H}^n}\backslash {B_R} \times {\mathbb{H}^n}\backslash {B_R}} \right).\]
Recalling the definition of the cut-off function $\varphi $, we have
\begin{align*}
   {I_1}& = c\int_{{B_R}} {\int_{{B_R}} {\frac{{{\varphi ^p}\left( \xi  \right){{\left| {v\left( \xi  \right) - v\left( \eta  \right)} \right|}^p}}}{{\| {{\eta ^{ - 1}} \circ \xi } \|_{{\mathbb{H}^n}}^{Q + sp}}}\,d\xi } d\eta }  + c\int_{{\mathbb{H}^n}\backslash {B_R}} {\int_{{B_R}} {\frac{{{\varphi ^p}\left( \xi  \right){{\left| {v\left( \xi  \right) - v\left( \eta  \right)} \right|}^p}}}{{\| {{\eta ^{ - 1}} \circ \xi } \|_{{\mathbb{H}^n}}^{Q + sp}}}\,d\xi } d\eta }  \\
   & \le c\int_{{B_R}} {\int_{{B_R}} {\frac{{{{\left| {v\left( \xi  \right) - v\left( \eta  \right)} \right|}^p}}}{{\| {{\eta ^{ - 1}} \circ \xi } \|_{{\mathbb{H}^n}}^{Q + sp}}}\,d\xi } d\eta }  + c\int_{{\mathbb{H}^n}\backslash {B_R}} {\int_{{B_{\frac{{R + r}}{2}}}} {\frac{{{{\left| {v\left( \xi  \right)} \right|}^p}}}{{\| {{\eta ^{ - 1}} \circ \xi } \|_{{\mathbb{H}^n}}^{Q + sp}}}\,d\xi } d\eta } ,
\end{align*}
where we naturally let $v\equiv0$ on ${{\mathbb{H}^n}\backslash {B_R}}$. For $\eta \in {{\mathbb{H}^n}\backslash {B_R}}$, $\xi \in {B_{\frac{{R + r}}{2}}}$, there holds that
\begin{align*}
   {\| {{\eta ^{ - 1}} \circ \xi_0 } \|_{{\mathbb{H}^n}}} & \le \left( {1 + \frac{{{{\| {{\xi ^{ - 1}} \circ {\xi _0}} \|}_{{\mathbb{H}^n}}}}}{{{{\| {{\eta ^{ - 1}} \circ \xi } \|}_{{\mathbb{H}^n}}}}}} \right){\| {{\eta ^{ - 1}} \circ \xi } \|_{{\mathbb{H}^n}}} \\
   &  \le \left(1 + \frac{(R+r)/2}{(R-r)/2} \right){\| {{\eta ^{ - 1}} \circ \xi } \|_{{\mathbb{H}^n}}} = \frac{{2R}}{{R - r}}{\| {{\eta ^{ - 1}} \circ \xi } \|_{{\mathbb{H}^n}}}.
\end{align*}
From this, it follows that
\begin{align*}
   \int_{{\mathbb{H}^n}\backslash {B_R}} {\int_{{B_{\frac{{R + r}}{2}}}} {\frac{{{{\left| {v\left( \xi  \right)} \right|}^p}}}{{\| {{\eta ^{ - 1}} \circ \xi } \|_{{\mathbb{H}^n}}^{Q + sp}}}\,d\xi } d\eta }& \le c{\left( {\frac{R}{{R - r}}} \right)^{Q + sp}}\int_{{\mathbb{H}^n}\backslash {B_R}} {\int_{{B_{\frac{{R + r}}{2}}}} {\frac{{{{\left| {v\left( \xi  \right)} \right|}^p}}}{{\| {{\eta ^{ - 1}} \circ {\xi _0}} \|_{{\mathbb{H}^n}}^{Q + sp}}}\,d\xi } d\eta }  \\
   &  \le c{\left( {\frac{R}{{R - r}}} \right)^{Q + sp}}{R^{ - sp}}\int_{{B_R}} {{{\left| {v\left( \xi  \right)} \right|}^p}\,d\xi } .
\end{align*}
For $I_2$, we derive
\begin{align*}
   {I_2} =& c\left( {\int_{{B_R}} {\int_{{B_R}} {} }  + \int_{{\mathbb{H}^n}\backslash {B_R}} {\int_{{B_R}} {} }  + \int_{{B_R}} {\int_{{\mathbb{H}^n}\backslash {B_R}} {} }  + \int_{{\mathbb{H}^n}\backslash {B_R}} {\int_{{\mathbb{H}^n}\backslash {B_R}} {} } } \right)\frac{{{{\left| {\varphi \left( \xi  \right) - \varphi \left( \eta  \right)} \right|}^p}{{\left| {\left( {v - {{\left( v \right)}_r}} \right)\left( \eta  \right)} \right|}^p}}}{{\| {{\eta ^{ - 1}} \circ \xi } \|_{{\mathbb{H}^n}}^{Q + sp}}}\,d\xi d\eta  \\
   \le& c\int_{{B_R}} {\int_{{B_R}} {\frac{{{{\left| {\varphi \left( \xi  \right) - \varphi \left( \eta  \right)} \right|}^p}{{\left| {\left( {v - {{\left( v \right)}_r}} \right)\left( \eta  \right)} \right|}^p}}}{{\| {{\eta ^{ - 1}} \circ \xi } \|_{{\mathbb{H}^n}}^{Q + sp}}}\,d\xi d\eta } }+ c\int_{{\mathbb{H}^n}\backslash {B_R}} {\int_{{B_{\frac{{R + r}}{2}}}} {\frac{{{{\left| {{{\left( v \right)}_r}} \right|}^p}}}{{\| {{\eta ^{ - 1}} \circ \xi } \|_{{\mathbb{H}^n}}^{Q + sp}}}\,d\xi d\eta } } \\
   & + \int_{{B_{\frac{{R + r}}{2}}}} {\int_{{\mathbb{H}^n}\backslash {B_R}} {\frac{{{{\left| {\left( {v - {{\left( v \right)}_r}} \right)\left( \eta  \right)} \right|}^p}}}{{\| {{\eta ^{ - 1}} \circ \xi } \|_{{\mathbb{H}^n}}^{Q + sp}}}\,d\xi d\eta } } \\
   : =& {I_{21}} + {I_{22}} + {I_{23}}.
\end{align*}

We first evalute $I_{21}$ as
\begin{align*}
   {I_{21}} & \le \frac{c}{{{{\left( {R - r} \right)}^p}}}\int_{{B_R}} {\int_{{B_R}} {\frac{{{{\left| { {v\left( \eta  \right) - {{\left( v \right)}_r}} } \right|}^p}}}{{\| {{\eta ^{ - 1}} \circ \xi } \|_{{\mathbb{H}^n}}^{Q + p\left( {s - 1} \right)}}}\,d\xi d\eta } }  \\
   &  \le \frac{c}{{{{\left( {R - r} \right)}^p}}}\int_{{B_R}} {{{\left| { {v\left( \eta  \right) - {{\left( v \right)}_r}} } \right|}^p}\int_{{B_{2R}}\left( \eta  \right)} {\frac{1}{{\| {{\eta ^{ - 1}} \circ \xi } \|_{{\mathbb{H}^n}}^{Q + p\left( {s - 1} \right)}}}\,d\xi d\eta } } \\
   & \le c{\left( {\frac{R}{{R - r}}} \right)^p}{R^{ - sp}}\int_{{B_R}} {{{\left| { {v\left( \eta  \right) - {{\left( v \right)}_r}} } \right|}^p}\,d\eta } \\
   & \le c{\left( {\frac{R}{{R - r}}} \right)^p}{R^{ - sp}}\left( {\int_{{B_R}} {{{\left| { {v\left( \eta  \right) - {{\left( v \right)}_R}} } \right|}^p}\,d\eta }  + \int_{{B_R}} {{{\left| { {{{\left( v \right)}_R} - {{\left( v \right)}_r}} } \right|}^p}\,d\eta } } \right)\\
    &\le c{\left( {\frac{R}{{R - r}}} \right)^p}{R^{ - sp}}\left( {{R^{sp}}\int_{{B_R}} {\int_{{B_R}} {\frac{{{{\left| { {v\left( \xi  \right) - v\left( \eta  \right)} } \right|}^p}}}{{\| {{\eta ^{ - 1}} \circ \xi } \|_{{\mathbb{H}^n}}^{Q + sp}}}\,d\xi d\eta } }  + {{\left| {{{\left( v \right)}_R} - {{\left( v \right)}_r}} \right|}^p}\left| {{B_R}} \right|} \right),
\end{align*}
where in the last line we have utilized Lemma \ref{Le101}. On the other hand,
\begin{align*}
   {\left| {{{\left( v \right)}_R} - {{\left( v \right)}_r}} \right|^p}\left| {{B_R}} \right|& = \left| {{B_R}} \right|{\left| {\fint_{{B_r}} {\left( {v - {{\left( v \right)}_R}} \right)d\xi } } \right|^p} \\
   &\le \left| {{B_R}} \right|\fint_{{B_r}} {{{\left| {v - {{\left( v \right)}_R}} \right|}^p}d\xi }  \\
  & \le  \frac{{\left| {{B_R}} \right|}}{{\left| {{B_r}} \right|}}\int_{{B_R}} {{{\left| {v - {{\left( v \right)}_R}} \right|}^p}d\xi }\\
  &  \le c{\left( {\frac{R}{r}} \right)^Q}{R^{sp}}\int_{{B_R}} {\int_{{B_R}} {\frac{{{{\left| { {v\left( \xi  \right) - v\left( \eta  \right)} } \right|}^p}}}{{\| {{\eta ^{ - 1}} \circ \xi } \|_{{\mathbb{H}^n}}^{Q + sp}}}d\xi d\eta } } .
\end{align*}
For $I_{22}$, we can see that
\begin{align*}
   {I_{22}}& \le c{\left( {\frac{R}{{R - r}}} \right)^{Q + sp}}\int_{{\mathbb{H}^n}\backslash {B_R}} {\int_{{B_{\frac{{R + r}}{2}}}} {\frac{{{{\left| {{{\left( v \right)}_r}} \right|}^p}}}{{\| {{\eta ^{ - 1}} \circ {\xi _0}} \|_{{\mathbb{H}^n}}^{Q + sp}}}d\xi d\eta } }  \\
   &  \le c{\left( {\frac{R}{{R - r}}} \right)^{Q + sp}}{R^{Q - sp}}\fint_{{B_r}} {{{\left| v \right|}^p}d\xi } \\
   & = c{\left( {\frac{R}{r}} \right)^Q}{\left( {\frac{R}{{R - r}}} \right)^{Q + sp}}{R^{ - sp}}\int_{{B_R}} {{{\left| v \right|}^p}d\xi } .
\end{align*}
Finally, the term $I_{23}$ is treated by
\begin{align*}
   {I_{23}}& \le c{\left( {\frac{R}{{R - r}}} \right)^{Q + sp}}\int_{{\mathbb{H}^n}\backslash {B_R}} {\frac{1}{{\| {{\xi ^{ - 1}} \circ {\xi _0}} \|_{{\mathbb{H}^n}}^{Q + sp}}}d\xi \int_{{B_{\frac{{R + r}}{2}}}} {{{\left| {v\left( \eta  \right) - {{\left( v \right)}_r}} \right|}^p}d\eta } }  \\
   & \le c\frac{{{R^Q}}}{{{{\left( {R - r} \right)}^{Q + sp}}}}\int_{{B_R}} {{{\left| {v\left( \eta  \right) - {{\left( v \right)}_r}} \right|}^p}d\eta } \\
   & \le c\frac{{{R^Q}}}{{{{\left( {R - r} \right)}^{Q + sp}}}}\left[{{R^{sp}}\int_{{B_R}} {\int_{{B_R}} {\frac{{{{\left| { {v\left( \xi  \right) - v\left( \eta  \right)} } \right|}^p}}}{{\| {{\eta ^{ - 1}} \circ \xi } \|_{{\mathbb{H}^n}}^{Q + sp}}}d\xi d\eta  + {\frac{{{R^{Q + sp}}}}{{{r^Q}}}} \int_{{B_R}} {\int_{{B_R}} {\frac{{{{\left| {{v\left( \xi  \right) - v\left( \eta  \right)} } \right|}^p}}}{{\| {{\eta ^{ - 1}} \circ \xi } \|_{{\mathbb{H}^n}}^{Q + sp}}}d\xi d\eta } } } } } \right],
\end{align*}
the procedure of which is analogous to $I_{21}$. Eventurally, after careful rearrangement, we obtain
\begin{align*}
    {\left( {\int_{{B_r}} {{{\left| {v - {{\left( v \right)}_r}} \right|}^{p_s^*}}d\xi } } \right)^{\frac{p}{{p_s^*}}}}
   &\le  c{R^Q}\left[ {1 + {{\left( {\frac{R}{{R - r}}} \right)}^p} + {{\left( {\frac{R}{r}} \right)}^Q}{{\left( {\frac{R}{{R - r}}} \right)}^p} + {{\left( {\frac{R}{{R - r}}} \right)}^{Q + sp}} + {{\left( {\frac{R}{r}} \right)}^Q}{{\left( {\frac{R}{{R - r}}} \right)}^{Q + sp}}} \right]\\
   &\quad\cdot \fint_{{B_R}} {\int_{{B_R}} {\frac{{{{\left| { {v\left( \xi  \right) - v\left( \eta  \right)} } \right|}^p}}}{{\| {{\eta ^{ - 1}} \circ \xi } \|_{{\mathbb{H}^n}}^{Q + sp}}}d\xi d\eta } }+ c{R^{Q - sp}}{\left( {\frac{R}{{R - r}}} \right)^{Q + sp}}\left[ {1 + {{\left( {\frac{R}{r}} \right)}^Q}} \right]\fint_{{B_R}} {{{\left| v \right|}^p}d\xi } ,
\end{align*}
which implies the statement.
\end{proof}

If we let $R=2r$ in the preceding Sobolev-Poincar\'{e} inequality, then we can get the very simple version below.
\begin{corollary}
Let $1 < p < \infty,s \in (0,1)$ fulfill $sp < Q$. Suppose that $v\in H{W^{s,p}}(B_{2r})$ and $B_r \subset B_{2r}$ are concentric balls. Then there exists a positive constant $c(n,p,s)$ such that
\[{\left( {{\fint_{{B_r}}}{{\left| {v - {{\left( v \right)}_{{r}}}} \right|}^{p_s^*}}\,d\xi } \right)^{\frac{p}{{p_s^*}}}} \le cr^{sp}{\fint_{B_{2r}}}\int_{B_{2r}} {\frac{{{{\left| {v\left( \xi  \right) - v\left( \eta  \right)} \right|}^p}}}{{\| {{\eta ^{ - 1}} \circ \xi } \|_{{\mathbb{H}^n}}^{Q + sp}}}\,d\xi } d\eta  + c{\fint_{B_{2r}}}{\left| v \right|^p}\,d\xi. \]
\end{corollary}

The following result shows an embedding relation between the fractional Sobolev spaces $HW^{t,q}(\Omega)$ and $HW^{s,p}(\Omega)$.

\begin{lemma}\label{Le105}
  Let $1< p \le q$ and $0<s<t<1$. Let also $\Omega$ be a bounded measurable subset of $\mathbb{H}^n$. Then there holds that, for each $v \in HW^{t,q}(\Omega)$,
  \[{\left( {\int_\Omega  {\int_\Omega  {\frac{{{{\left| {{v\left( \xi  \right) - v\left( \eta  \right)} } \right|}^p}}}{{\| {{\eta ^{ - 1}} \circ \xi } \|_{{\mathbb{H}^n}}^{Q + sp}}}d\xi d\eta } } } \right)^{\frac{1}{p}}} \le c{\left| \Omega  \right|^{\frac{{q - p}}{{pq}}}}{\left( {\mathrm{diam}\left( \Omega  \right)} \right)^{t - s}}{\left( {\int_\Omega  {\int_\Omega  {\frac{{{{\left| { {v\left( \xi  \right) - v\left( \eta  \right)} } \right|}^q}}}{{\| {{\eta ^{ - 1}} \circ \xi } \|_{{\mathbb{H}^n}}^{Q + tq}}}d\xi d\eta } } } \right)^{\frac{1}{q}}},\]
  where $c>0$ depends upon $n,p,q,s,t.$
\end{lemma}

\begin{proof} For $p<q$, we first utilize the H\"{o}lder inequality to get
\begin{align*}
    \int_\Omega  {\int_\Omega  {\frac{{{{\left| { {v\left( \xi  \right) - v\left( \eta  \right)}} \right|}^p}}}{{\| {{\eta ^{ - 1}} \circ \xi } \|_{{\mathbb{H}^n}}^{Q + sp}}}\,d\xi d\eta } }
 & =   \int_\Omega  {\int_\Omega  {\frac{{{{\left| { {v\left( \xi  \right) - v\left( \eta  \right)} } \right|}^p}}}{{\| {{\eta ^{ - 1}} \circ \xi } \|_{{\mathbb{H}^n}}^{\left( {Q + tq} \right)\frac{p}{q}}}}\frac{1}{{\| {{\eta ^{ - 1}} \circ \xi } \|_{{\mathbb{H}^n}}^{Q\frac{{q - p}}{q} + (s - t)p}}}\,d\xi d\eta } } \\
  & \le {\left( {\int_\Omega  {\int_\Omega  {\frac{{{{\left| {{v\left( \xi  \right) - v\left( \eta  \right)} } \right|}^q}}}{{\| {{\eta ^{ - 1}} \circ \xi } \|_{{\mathbb{H}^n}}^{Q + tq}}}\,d\xi d\eta } } } \right)^{\frac{p}{q}}}{\left( {\int_\Omega  {\int_\Omega  {\frac{1}{{\| {{\eta ^{ - 1}} \circ \xi } \|_{{\mathbb{H}^n}}^{Q + \frac{{(s - t)pq}}{{q - p}}}}}\,d\xi d\eta } } } \right)^{\frac{{q - p}}{q}}}.
\end{align*}
On the other hand,
\begin{align*}
   \int_\Omega  {\int_\Omega  {\frac{1}{{\| {{\eta ^{ - 1}} \circ \xi } \|_{{\mathbb{H}^n}}^{Q + \frac{{(s - t)pq}}{{q - p}}}}}\,d\xi d\eta } } & \le \int_\Omega  {\int_{{B_d}\left( \eta  \right)} {\frac{1}{{\| {{\eta ^{ - 1}} \circ \xi } \|_{{\mathbb{\mathbb{H}}^n}}^{Q + \frac{{(s - t)pq}}{{q - p}}}}}\,d\xi d\eta } }\\
  &\le  Q\left| {{B_1}} \right|\int_\Omega  {\int_0^d {{\rho ^{\frac{{(t - s)pq}}{{q - p}} - 1}}\,d\rho d\eta } }  = \frac{{Q\left| {{B_1}} \right|\left( {q - p} \right)}}{{(t - s)pq}}{d^{\frac{{(t - s)pq}}{{q - p}}}}\left| \Omega  \right|
\end{align*}
with $d: =\mathrm{ diam}\left( \Omega  \right)$. The combination of preceding inequalities implies the desired displaying.

If $q=p$, noting ${\| {{\eta ^{ - 1}} \circ \xi } \|_{{\mathbb{H}^n}}} \le \mathrm{diam}\left( \Omega  \right)$ for $\xi,\eta \in \Omega$ and $s<t$, we can readily obtain
\begin{align*}
   {\left( {\int_\Omega  {\int_\Omega  {\frac{{{{\left| {{v\left( \xi  \right) - v\left( \eta  \right)} } \right|}^p}}}{{\| {{\eta ^{ - 1}} \circ \xi } \|_{{\mathbb{H}^n}}^{Q + sp}}}\,d\xi d\eta } } } \right)^{\frac{1}{p}}}
   &=  {\left( {\int_\Omega  {\int_\Omega  {\frac{{{{\left| { {v\left( \xi  \right) - v\left( \eta  \right)} } \right|}^p}}}{{\| {{\eta ^{ - 1}} \circ \xi } \|_{{\mathbb{H}^n}}^{Q + tp}}}\frac{1}{{\| {{\eta ^{ - 1}} \circ \xi } \|_{{\mathbb{H}^n}}^{(s - t)p}}}\,d\xi d\eta } } } \right)^{\frac{1}{p}}}\\
   & \le {\left( {\mathrm{diam}\left( \Omega  \right)} \right)^{t - s}}{\left( {\int_\Omega  {\int_\Omega  {\frac{{{{\left| { {v\left( \xi  \right) - v\left( \eta  \right)} } \right|}^p}}}{{\| {{\eta ^{ - 1}} \circ \xi } \|_{{\mathbb{H}^n}}^{Q + tp}}}\,d\xi d\eta } } } \right)^{\frac{1}{p}}}.
\end{align*}
Now we complete the proof.
\end{proof}

The forthcoming two lemmas are the consequences of these results above, which will be exploited in the proof of boundedness and H\"{o}lder continuity for solutions.

\begin{lemma}\label{Le24}
Assume that the constant $s,\;t,\;p$ and $q$ satisfy \eqref{eqy18} and \eqref{eqy112}. Then for every $f \in H{W^{s,p}}\left( {{B_r}} \right)$ we infer that
\begin{align*}
 & \fint_{{B_r}} {\left( {{{\left| {\frac{f}{{{r^s}}}} \right|}^p} + {a_0}{{\left| {\frac{f}{{{r^t}}}} \right|}^q}} \right)\,d\xi } \\
   \le& c{a_0}\frac{{D_1^{\frac{q}{p}}(R,r)}}{{{r^{tq}}}}{\left( {\fint_{{B_R}} {\int_{{B_R}} {\frac{{{{\left| {f\left( \xi  \right) - f\left( \eta  \right)} \right|}^p}}}{{\| {{\eta ^{ - 1}} \circ \xi } \|_{{\mathbb{H}^n}}^{Q + sp}}}\,d\xi } d\eta } } \right)^{\frac{q}{p}}}\\
   &  + c\frac{{{D_1}(R,r)}}{{{r^{sp}}}}{\left( {\frac{{\left| {{\rm{supp}}\,f} \right|}}{{\left| {{B_r}} \right|}}} \right)^{\frac{{sp}}{Q}}}\fint_{{B_R}} {\int_{{B_R}} {\frac{{{{\left| {f\left( \xi  \right) - f\left( \eta  \right)} \right|}^p}}}{{\| {{\eta ^{ - 1}} \circ \xi } \|_{{\mathbb{H}^n}}^{Q + sp}}}\,d\xi } d\eta }  \\
   &  + c\left[ {{{\left(\frac{R}{r}\right)^Q\left( {\frac{{\left| {{\rm{supp}}\,f} \right|}}{{\left| {{B_r}} \right|}}} \right)}^{p - 1}} + D_2^{\frac{q}{p}}(R,r)\left( {{{\left( {\frac{{\left| {{\rm{supp}}\,f} \right|}}{{\left| {{B_r}} \right|}}} \right)}^{\frac{{sp}}{Q}}} + {{\left( {\frac{{\left| \widetilde{{\rm{supp}}}\, f \right|}}{{\left| {{B_R}} \right|}}} \right)}^{\frac{{q - p}}{p}}}} \right)} \right]\fint_{{B_R}} {\left( {{{\left| {\frac{f}{{{r^s}}}} \right|}^p} + {a_0}{{\left| {\frac{f}{{{r^t}}}} \right|}^q}} \right)\,d\xi },
\end{align*}
where ${\rm{supp}}\,f: = \{ {B_r}:f \ne 0\} $ and $\widetilde{\rm{supp}}\,f: = \{ {B_R}:f \ne 0\} $, and $c>0$ depends only upon $n,p,q,s,t$, and ${a_0}$ is any positive constant.
\end{lemma}

\begin{proof} By the H\"{o}lder inequality and Proposition \ref{Pro104}, we obtain
\begin{align*}
   \fint_{{B_r}} {{{\left| {\frac{f}{{{r^s}}}} \right|}^p}\,d\xi }  &\le c\fint_{{B_r}} {{{\left| {\frac{{f - {{\left( f \right)}_r}}}{{{r^s}}}} \right|}^p}\,d\xi }  + c{\left| {\frac{{{{\left( f \right)}_r}}}{{{r^s}}}} \right|^p} \\
&\le     c{\left( {\frac{{\left| {{\rm{supp}}\,f} \right|}}{{\left| {{B_r}} \right|}}} \right)^{\frac{{sp}}{Q}}}{\left( {\fint_{{B_r}} {{{\left| {\frac{{f - {{\left( f \right)}_r}}}{{{r^s}}}} \right|}^{p_s^*}}\,d\xi } } \right)^{\frac{p}{{p_s^*}}}} + c{\left| {\frac{{{{\left( f \right)}_r}}}{{{r^s}}}} \right|^p}\\
 &\le c\frac{{{D_1}(R,r)}}{{{r^{sp}}}}{\left( {\frac{{\left| {{\rm{supp}}\,f} \right|}}{{\left| {{B_r}} \right|}}} \right)^{\frac{{sp}}{Q}}}\fint_{{B_R}} {\int_{{B_R}} {\frac{{{{\left| {f\left( \xi  \right) - f\left( \eta  \right)} \right|}^p}}}{{\| {{\eta ^{ - 1}} \circ \xi } \|_{{\mathbb{H}^n}}^{Q + sp}}}\,d\xi } d\eta } \\
 &\quad+ c{\left( {\frac{{\left| {{\rm{supp}}\,f} \right|}}{{\left| {{B_r}} \right|}}} \right)^{\frac{{sp}}{Q}}}{D_2}(R,r)\fint_{{B_R}} {{{\left| {\frac{f}{{{r^s}}}} \right|}^p}\,d\xi }  + c{\left( {\frac{{\left| {{\rm{supp}}\,f} \right|}}{{\left| {{B_r}} \right|}}} \right)^{p - 1}}\fint_{{B_r}} {{{\left| {\frac{f}{{{r^s}}}} \right|}^p}\,d\xi },
\end{align*}
where we used the inequality below,
\[{\left| {\frac{{{{\left( f \right)}_r}}}{{{r^s}}}} \right|^p} = {r^{ - sp}}{\left| {\fint_{{B_r}} {f{\chi _{\left\{ {f \ne 0} \right\}}}\,d\xi } } \right|^p} \le {\left( {\frac{{\left| {{\rm{supp}}\;f} \right|}}{{\left| {{B_r}} \right|}}} \right)^{p - 1}}\fint_{{B_r}} {{{\left| {\frac{f}{{{r^s}}}} \right|}^p}\,d\xi } .\]

On the other hand, via the H\"{o}lder inequality and Proposition \ref{Pro104} again,
\begin{align*}
   \fint_{{B_r}} {{{\left| {\frac{f}{{{r^t}}}} \right|}^q}d\xi }  \le& c{\left( {\fint_{{B_r}} {{{\left| {\frac{{f - {{\left( f \right)}_r}}}{{{r^t}}}} \right|}^{p_s^*}}d\xi } } \right)^{\frac{q}{{p_s^*}}}} + c{\left| {\frac{{{{\left( f \right)}_r}}}{{{r^t}}}} \right|^q} \\
   \le&  c\frac{{D_1^{\frac{q}{p}}(R,r)}}{{{r^{tq}}}}{\left( {\fint_{{B_R}} {\int_{{B_R}} {\frac{{{{\left| {f\left( \xi  \right) - f\left( \eta  \right)} \right|}^p}}}{{\| {{\eta ^{ - 1}} \circ \xi } \|_{{\mathbb{H}^n}}^{Q + sp}}}\,d\xi } d\eta } } \right)^{\frac{q}{p}}}+ cD_2^{\frac{q}{p}}(R,r){\left( {\fint_{{B_R}} {{{\left| {\frac{f}{{{r^t}}}} \right|}^p}\,d\xi } } \right)^{\frac{q}{p}}} + c{\left| {\frac{{{{\left( f \right)}_r}}}{{{r^t}}}} \right|^q}\\
   \le&  c\frac{{D_1^{\frac{q}{p}}(R,r)}}{{{r^{tq}}}}{\left( {\fint_{{B_R}} {\int_{{B_R}} {\frac{{{{\left| {f\left( \xi  \right) - f\left( \eta  \right)} \right|}^p}}}{{\| {{\eta ^{ - 1}} \circ \xi } \|_{{\mathbb{H}^n}}^{Q + sp}}}\,d\xi } d\eta } } \right)^{\frac{q}{p}}}\\
   & + cD_2^{\frac{q}{p}}(R,r){\left( {\frac{{\left| \widetilde{\rm{supp }}\, f \right|}}{{\left| {{B_R}} \right|}}} \right)^{\frac{{q - p}}{p}}}\fint_{{B_R}} {{{\left| {\frac{f}{{{r^t}}}} \right|}^q}\,d\xi }  + c{\left( {\frac{{\left| {{\rm{supp}}\,f} \right|}}{{\left| {{B_r}} \right|}}} \right)^{p - 1}}\fint_{{B_r}} {{{\left| {\frac{f}{{{r^t}}}} \right|}^q}\,d\xi },
\end{align*}
where we can see that
\[{\left| {\frac{{{{\left( f \right)}_r}}}{{{r^t}}}} \right|^q} \le {\left( {\frac{{\left| {{\rm{supp}}\,f} \right|}}{{\left| {{B_r}} \right|}}} \right)^{q - 1}}\fint_{{B_r}} {{{\left| {\frac{f}{{{r^t}}}} \right|}^q}\,d\xi }  \le {\left( {\frac{{\left| {{\rm{supp}}\,f} \right|}}{{\left| {{B_r}} \right|}}} \right)^{p - 1}}\fint_{{B_r}} {{{\left| {\frac{f}{{{r^t}}}} \right|}^q}\,d\xi } \]
and
\[{r^{ - tq}}{\left( {\fint_{{B_R}} {{{\left| f \right|}^p}\,d\xi } } \right)^{\frac{q}{p}}} \le {\left( {\frac{{\left| \widetilde{\rm{supp}}\,f \right|}}{{\left| {{B_R}} \right|}}} \right)^{\frac{{q - p}}{p}}}\fint_{{B_R}} {{{\left| {\frac{f}{{{r^t}}}} \right|}^q}\,d\xi } .\]
We finally observe the plain relation that
$$
\fint_{{B_r}}\left|\frac{f}{r^s}\right|^p+a_0\left|\frac{f}{r^t}\right|^q\,d\xi\le c\left(\frac{R}{r}\right)^Q\fint_{{B_R}}\left|\frac{f}{r^s}\right|^p+a_0\left|\frac{f}{r^t}\right|^q\,d\xi.
$$
In summary, we combine all the previous inequalities to arrive at the desired display.
\end{proof}

Now denote
\[a_R^ + : = \mathop {\sup }\limits_{{B_R} \times {B_R}} a\left( {\cdot,\cdot} \right) \quad \text{and} \quad a_R^ - : = \mathop {\inf }\limits_{{B_R} \times {B_R}} a\left( {\cdot,\cdot} \right).\]

\begin{lemma}\label{Le107}
Let $s,t \in (0,1)$, $1<p\le q$ and $a(\cdot,\cdot)$ satisfies \eqref{eqy110} with
\begin{equation}\label{eqy113}
 tq \le sp + \alpha .
\end{equation}
Assume $f \in H{W^{t,q}}\left( {{B_{\bar R}}} \right) \cap {L^\infty }\left( {{B_{\bar R}}} \right)$ with ${\bar R} \le 1$. Then for $\gamma : = \min \left\{ {\frac{{p_s^*}}{p},\frac{{q_t^*}}{q}} \right\} > 1$, we have
\begin{align*}
   & {\left[ {\fint_{{B_r}} {{{\left( {{{\left| {\frac{f}{{{r^s}}}} \right|}^p} + a_{\bar R}^ + {{\left| {\frac{f}{{{r^t}}}} \right|}^q}} \right)}^\gamma }\,d\xi } } \right]^{\frac{1}{\gamma }}} \\
    \le& c\left( {1 + \| f \|_{{L^\infty }\left( {{B_r}} \right)}^{q - p}} \right)\left( {\frac{{{D_1}\left( {R,r} \right)}}{{{r^{sp}}}} + \frac{{{\widetilde D}_1\left( {R,r} \right)}}{{{r^{tq}}}}} \right)\fint_{{B_R}} {\int_{{B_R}} {\frac{{H\left( {\xi ,\eta ,\left| {f\left( \xi  \right) - f\left( \eta  \right)} \right|} \right)}}{{\| {{\eta ^{ - 1}}\circ \xi } \|_{{\mathbb{H}^n}}^Q}}\,d\xi d\eta } } \\
    & + c\left( {1 + \| f \|_{{L^\infty }\left( {{B_r}} \right)}^{q - p}} \right)\left( {1 + {D_2}\left( {R,r} \right) + {\widetilde D}_2\left( {R,r} \right)} \right)\fint_{{B_R}} {\left( {{{\left| {\frac{f}{{{r^s}}}} \right|}^p} + a_{\bar R}^ - {{\left| {\frac{f}{{{r^t}}}} \right|}^q}} \right)\,d\xi } .
\end{align*}
where ${B_r} \subset {B_R} \subseteq {B_{\bar R}}$ are concentric balls with $\frac{1}{2}\bar R \le r < R \le \bar R$, and $c>0$ depends only on $n,p,q,s,t$ and ${\left[ a \right]_\alpha }$. Here ${\widetilde D}_1(R,r)$ and ${\widetilde D}_2(R,r)$ are seperately ${D_1}\left( {R,r} \right)$ and ${D_2}\left( {R,r} \right) $ replaced $sp$ by $tq$.
\end{lemma}

\begin{proof} In view of H\"{o}lder continuity of $a$, we have
\[a_{\bar R}^ +  \le a_{\bar R}^ -  + 4{\left[ a \right]_\alpha }{{\bar R}^\alpha } \le a_{\bar R}^ -  + 8{\left[ a \right]_\alpha }{r^\alpha }.\]
Then we by employing $tq \le sp+\alpha$, $r \le 1$ have
\[a_{\bar R}^ + {\left| {\frac{f}{{{r^t}}}} \right|^q} \le a_{\bar R}^ - {\left| {\frac{f}{{{r^t}}}} \right|^q} + c{r^{\alpha  - tq + sp}}{\left| f \right|^{q - p}}{\left| {\frac{f}{{{r^s}}}} \right|^p}.\]
Thus
\begin{align*}
{\left[ {\fint_{{B_r}} {{{\left( {{{\left| {\frac{f}{{{r^s}}}} \right|}^p} + a_{\bar R}^ + {{\left| {\frac{f}{{{r^t}}}} \right|}^q}} \right)}^\gamma }\,d\xi } } \right]^{\frac{1}{\gamma }}}
   &\le c\left( {1 + \| f \|_{{L^\infty }\left( {{B_r}} \right)}^{q - p}} \right){\left[ {\fint_{{B_r}} {{{\left( {{{\left| {\frac{f}{{{r^s}}}} \right|}^p} + a_{\bar R}^ - {{\left| {\frac{f}{{{r^t}}}} \right|}^q}} \right)}^\gamma }\,d\xi } } \right]^{\frac{1}{\gamma }}}\\
    &\le c\left( {1 + \| f \|_{{L^\infty }\left( {{B_r}} \right)}^{q - p}} \right){\left[ {\fint_{{B_r}} {{{\left( {{{\left| {\frac{{f - {{\left( f \right)}_r}}}{{{r^s}}}} \right|}^p} + a_{\bar R}^ - {{\left| {\frac{{f - {{\left( f \right)}_r}}}{{{r^t}}}} \right|}^q}} \right)}^\gamma }\,d\xi } } \right]^{\frac{1}{\gamma }}}\\
    &\quad + c\left( {1 + \| f \|_{{L^\infty }\left( {{B_r}} \right)}^{q - p}} \right)\left( {{{\left| {\frac{{{{\left( f \right)}_r}}}{{{r^s}}}} \right|}^p} + a_{\bar R}^ - {{\left| {\frac{{{{\left( f \right)}_r}}}{{{r^t}}}} \right|}^q}} \right).
\end{align*}
Observe that
\[{\left| {\frac{{{{\left( f \right)}_r}}}{{{r^s}}}} \right|^p} + a_{\bar R}^ - {\left| {\frac{{{{\left( f \right)}_r}}}{{{r^t}}}} \right|^q} \le \fint_{{B_r}} {{{\left( {{{\left| {\frac{f}{{{r^s}}}} \right|}^p} + a_{\bar R}^ - {{\left| {\frac{f}{{{r^t}}}} \right|}^q}} \right)} }\,d\xi } .\]
Moreover, it follows from Proposition \ref{Pro104} that
\begin{align*}
    {\left( {\fint_{{B_r}} {{{\left| {\frac{{f - {{\left( f \right)}_r}}}{{{r^s}}}} \right|}^{p\gamma }}\,d\xi } } \right)^{\frac{1}{\gamma }}} &\le {\left( {\fint_{{B_r}} {{{\left| {\frac{{f - {{\left( f \right)}_r}}}{{{r^s}}}} \right|}^{p_s^*}}\,d\xi } } \right)^{\frac{p}{{p_s^*}}}} \\
   &\le  c\frac{{{D_1}\left( {R,r} \right)}}{{{r^{sp}}}}\fint_{{B_R}} {\int_{{B_R}} {\frac{{{{\left| {f\left( \xi  \right) - f\left( \eta  \right)} \right|}^p}}}{{\| {{\eta ^{ - 1}}\circ \xi } \|_{{\mathbb{H}^n}}^{Q + sp}}}\,d\xi d\eta } }  + c{D_2}\left( {R,r} \right)\fint_{{B_R}} {{{\left| {\frac{f}{{{r^s}}}} \right|}^p}\,d\xi }
\end{align*}
and
\begin{align*}
   {\left( {\fint_{{B_r}} {{{\left| {\frac{{f - {{\left( f \right)}_r}}}{{{r^t}}}} \right|}^{q\gamma }}d\xi } } \right)^{\frac{1}{\gamma }}} &\le {\left( {\fint_{{B_r}} {{{\left| {\frac{{f - {{\left( f \right)}_r}}}{{{r^t}}}} \right|}^{q_t^*}}\,d\xi } } \right)^{\frac{q}{{q_t^*}}}} \\
   &\le  c\frac{{{{\widetilde D}_1}\left( {R,r} \right)}}{{{r^{tq}}}}\fint_{{B_R}} {\int_{{B_R}} {\frac{{{{\left| {f\left( \xi  \right) - f\left( \eta  \right)} \right|}^q}}}{{\| {{\eta ^{ - 1}}\circ \xi } \|_{{\mathbb{H}^n}}^{Q + tq}}}\,d\xi d\eta } }  + c{{\widetilde D}_2}\left( {R,r} \right)\fint_{{B_R}} {{{\left| {\frac{f}{{{r^t}}}} \right|}^q}\,d\xi }.
\end{align*}
Merging the last four inequality leads to
\begin{align*}
   & {\left[ {\fint_{{B_r}} {{{\left( {{{\left| {\frac{f}{{{r^s}}}} \right|}^p} + a_{\bar R}^ + {{\left| {\frac{f}{{{r^t}}}} \right|}^q}} \right)}^\gamma }\,d\xi } } \right]^{\frac{1}{\gamma }}} \\
   \le& c\left( {1 + \| f \|_{{L^\infty }\left( {{B_r}} \right)}^{q - p}} \right)\left( {\frac{{{D_1}\left( {R,r} \right)}}{{{r^{sp}}}} + \frac{{{{\widetilde D}_1}\left( {R,r} \right)}}{{{r^{tq}}}}} \right)\fint_{{B_R}} {\int_{{B_R}} {\left( {\frac{{{{\left| {f\left( \xi  \right) - f\left( \eta  \right)} \right|}^p}}}{{\| {{\eta ^{ - 1}}\circ \xi } \|_{{\mathbb{H}^n}}^{Q + sp}}} + a_{\bar R}^ - \frac{{{{\left| {f\left( \xi  \right) - f\left( \eta  \right)} \right|}^q}}}{{\| {{\eta ^{ - 1}}\circ \xi } \|_{{\mathbb{H}^n}}^{Q + tq}}}} \right)\,d\xi d\eta } } \\
   & + c\left( {1 + \| f \|_{{L^\infty }\left( {{B_r}} \right)}^{q - p}} \right)\left( {1 + {D_2}\left( {R,r} \right) + {{\widetilde D}_2}\left( {R,r} \right)} \right)\fint_{{B_R}} {\left( {{{\left| {\frac{f}{{{r^s}}}} \right|}^p} + a_{\bar R}^ - {{\left| {\frac{f}{{{r^t}}}} \right|}^q}} \right)\,d\xi } \\
    \le& c\left( {1 + \| f \|_{{L^\infty }\left( {{B_r}} \right)}^{q - p}} \right)\left( {\frac{{{D_1}\left( {R,r} \right)}}{{{r^{sp}}}} + \frac{{{{\widetilde D}_1}\left( {R,r} \right)}}{{{r^{tq}}}}} \right)\fint_{{B_R}} {\int_{{B_R}} {\frac{{H\left( {\xi ,\eta ,\left| {f\left( \xi  \right) - f\left( \eta  \right)} \right|} \right)}}{{\| {{\eta ^{ - 1}}\circ \xi } \|_{{\mathbb{H}^n}}^Q}}\,d\xi d\eta } } \\
    & + c\left( {1 + \| f \|_{{L^\infty }\left( {{B_r}} \right)}^{q - p}} \right)\left( {1 + {D_2}\left( {R,r} \right) + {{\widetilde D}_2}\left( {R,r} \right)} \right)\fint_{{B_R}} {\left( {{{\left| {\frac{f}{{{r^s}}}} \right|}^p} + a_{\bar R}^ - {{\left| {\frac{f}{{{r^t}}}} \right|}^q}} \right)\,d\xi } .
\end{align*}
We now finish the proof.
\end{proof}

\section{Local boundedness}
\label{Section 4}

This section is devoted to showing the interior boundedness of weak solutions to Eq. (1.1) by means of utilizing the key ingredient, a Caccioppoli-type inequality in the nonlocal framework. The forthcoming lemma indicates the multiplication of each function in $\mathcal{A}(\Omega)$ and a cut-off function also belongs to $\mathcal{A}(\Omega)$.

\begin{lemma}\label{Le41}
Let $s,\;t,\;p$ and $q$ satisfy \eqref{eqy18} and $\varphi  \in HW_0^{1,\infty }\left( {{B_r}} \right),\;v \in {\mathcal A}(\Omega)$. If one of the following two conditions holds:\\
(i) The inequality \eqref{eqy112} holds and $v  \in {L^p}\left( {{B_{2r}}} \right)$ satisfies $\rho \left( {v ;{B_{2r}}} \right) < \infty $;\\
(ii) $v  \in {L^q}\left( {{B_{2r}}} \right)$ satisfies $\rho \left( {v ;{B_{2r}}} \right) < \infty $, \\
then $\rho \left( {v \varphi ;\mathbb{H}^n} \right) < \infty $. In particular, $v \varphi \in {\mathcal A}(\Omega)$ whenever $ {B_{2r}}\subset\Omega$.
\end{lemma}

\begin{proof} By $v \in {\mathcal A}(\Omega)$, Proposition \ref{Pro104} and \eqref{eqy112}, we get $v  \in {L^q}\left( {{B_{2r}}} \right)$ in (i). Thus, we just consider condition (ii). By the definition of $\rho \left( {v \varphi ;\mathbb{H}^n} \right)$, we have
\begin{align}\label{eq44}
   \rho \left( {v \varphi ;{\mathbb{H}^n}} \right) =& \int_{{\mathbb{H}^n}} {\int_{{\mathbb{H}^n}} {H\left( {\xi ,\eta ,\left| {v \left( \xi  \right)\varphi \left( \xi  \right) - v \left( \eta  \right)\varphi \left( \eta  \right)} \right|} \right)\frac{{d\xi d\eta }}{{\| {{\eta ^{ - 1}} \circ \xi } \|_{{\mathbb{H}^n}}^Q}}} } \nonumber \\
   =& \int_{{\mathbb{H}^n}} {\int_{{\mathbb{H}^n}\backslash {B_{2r}}} {H\left( {\xi ,\eta ,\left| {v \left( \xi  \right)\varphi \left( \xi  \right) - v \left( \eta  \right)\varphi \left( \eta  \right)} \right|} \right)\frac{{d\xi d\eta }}{{\| {{\eta ^{ - 1}} \circ \xi } \|_{{\mathbb{H}^n}}^Q}}} }\nonumber \\
   & + \int_{{\mathbb{H}^n}} {\int_{{B_{2r}}} {H\left( {\xi ,\eta ,\left| {v \left( \xi  \right)\varphi \left( \xi  \right) - v \left( \eta  \right)\varphi \left( \eta  \right)} \right|} \right)\frac{{d\xi d\eta }}{{\| {{\eta ^{ - 1}} \circ \xi } \|_{{\mathbb{H}^n}}^Q}}} }\nonumber \\
    = &\int_{{\mathbb{H}^n}\backslash {B_{2r}}} {\int_{{\mathbb{H}^n}\backslash {B_{2r}}} {H\left( {\xi ,\eta ,\left| {v \left( \xi  \right)\varphi \left( \xi  \right) - v \left( \eta  \right)\varphi \left( \eta  \right)} \right|} \right)\frac{{d\xi d\eta }}{{\| {{\eta ^{ - 1}} \circ \xi } \|_{{\mathbb{H}^n}}^Q}}} }\nonumber \\
   & + 2\int_{{\mathbb{H}^n}\backslash {B_{2r}}} {\int_{{B_{2r}}} {H\left( {\xi ,\eta ,\left| {v \left( \xi  \right)\varphi \left( \xi  \right) - v \left( \eta  \right)\varphi \left( \eta  \right)} \right|} \right)\frac{{d\xi d\eta }}{{\| {{\eta ^{ - 1}} \circ \xi } \|_{{\mathbb{H}^n}}^Q}}} }\nonumber \\
   & + \int_{{B_{2r}}} {\int_{{B_{2r}}} {H\left( {\xi ,\eta ,\left| {v \left( \xi  \right)\varphi \left( \xi  \right) - v \left( \eta  \right)\varphi \left( \eta  \right)} \right|} \right)\frac{{d\xi d\eta }}{{\| {{\eta ^{ - 1}} \circ \xi } \|_{{\mathbb{H}^n}}^Q}}} } \nonumber \\
    = :&{I_1} + 2{I_2} + {I_3}.
\end{align}
Owing to $\varphi  \in HW_0^{1,\infty }\left( {{B_r}} \right)$, we find
\begin{equation}\label{eq45}
{I_1} = \int_{{\mathbb{H}^n}\backslash {B_{2r}}} {\int_{{\mathbb{H}^n}\backslash {B_{2r}}} {H\left( {\xi ,\eta ,\left| {v \left( \xi  \right)\varphi \left( \xi  \right) - v \left( \eta  \right)\varphi \left( \eta  \right)} \right|} \right)\frac{{d\xi d\eta }}{{\| {{\eta ^{ - 1}} \circ \xi } \|_{{\mathbb{H}^n}}^Q}}} }  = 0
\end{equation}
and
\begin{align}\label{eq46}
   {I_2} =& \int_{{\mathbb{H}^n}\backslash {B_{2r}}} {\int_{{B_{2r}}} {H\left( {\xi ,\eta ,\left| {v \left( \xi  \right)\varphi \left( \xi  \right) - v \left( \eta  \right)\varphi \left( \eta  \right)} \right|} \right)\frac{{d\xi d\eta }}{{\| {{\eta ^{ - 1}} \circ \xi } \|_{{\mathbb{H}^n}}^Q}}} }  \nonumber\\
    =& \int_{{\mathbb{H}^n}\backslash {B_{2r}}} {\int_{{B_{2r}}} {H\left( {\xi ,\eta ,\left| {v \left( \xi  \right)\varphi \left( \xi  \right)} \right|} \right)\frac{{d\xi d\eta }}{{\| {{\eta ^{ - 1}} \circ \xi } \|_{{\mathbb{H}^n}}^Q}}} } \nonumber\\
   = &\int_{{\mathbb{H}^n}\backslash {B_{2r}}} {\int_{{B_{2r}}} {\left( {\frac{{{{\left| {v \left( \xi  \right)\varphi \left( \xi  \right)} \right|}^p}}}{{\| {{\eta ^{ - 1}} \circ \xi } \|_{{\mathbb{H}^n}}^{Q + sp}}} + a\left( {\xi ,\eta } \right)\frac{{{{\left| {v \left( \xi  \right)\varphi \left( \xi  \right)} \right|}^q}}}{{\| {{\eta ^{ - 1}} \circ \xi } \|_{{\mathbb{H}^n}}^{Q + tq}}}} \right)\;d\xi d\eta } } \nonumber\\
    \le& {\left( {{{\| \varphi  \|}_{{L^\infty }\left( {{B_r}} \right)}} + 1} \right)^q}\int_{{\mathbb{H}^n}\backslash {B_{2r}}} {\int_{{B_r}} {\left( {\frac{{{{\left| {v \left( \xi  \right)} \right|}^p}}}{{\| {{\eta ^{ - 1}} \circ \xi } \|_{{\mathbb{H}^n}}^{Q + sp}}} + {{\| a \|}_{{L^\infty }}}\frac{{{{\left| {v \left( \xi  \right)} \right|}^q}}}{{\| {{\eta ^{ - 1}} \circ \xi } \|_{{\mathbb{H}^n}}^{Q + tq}}}} \right)\;d\xi d\eta } } \nonumber\\
     \le& c{\left( {{{\| \varphi  \|}_{{L^\infty }\left( {{B_r}} \right)}} + 1} \right)^q}\left( {{r^{ - sp}}\int_{{B_r}} {{{\left| {v \left( \xi  \right)} \right|}^p}\;d\xi }  + {{\| a \|}_{{L^\infty }}}{r^{ - tq}}\int_{{B_r}} {{{\left| {v \left( \xi  \right)} \right|}^p}\;d\xi } } \right) < \infty .
\end{align}
The term $I_3$ is estimated as
\begin{align}\label{eq47}
   {I_3} =& \int_{{B_{2r}}} {\int_{{B_{2r}}} {H\left( {\xi ,\eta ,\left| {v \left( \xi  \right)\varphi \left( \xi  \right) - v \left( \eta  \right)\varphi \left( \eta  \right)} \right|} \right)\frac{{d\xi d\eta }}{{\| {{\eta ^{ - 1}} \circ \xi } \|_{{\mathbb{H}^n}}^Q}}} } \nonumber \\
   \le& c\int_{{B_{2r}}} {\int_{{B_{2r}}} {H\left( {\xi ,\eta ,\left| {\left( {v \left( \xi  \right) - v \left( \eta  \right)} \right)\varphi \left( \eta  \right)} \right|} \right)\frac{{d\xi d\eta }}{{\| {{\eta ^{ - 1}} \circ \xi } \|_{{\mathbb{H}^n}}^Q}}} }\nonumber \\
   & + c\int_{{B_{2r}}} {\int_{{B_{2r}}} {H\left( {\xi ,\eta ,\left| {v \left( \xi  \right)\left( {\varphi \left( \xi  \right) - \varphi \left( \eta  \right)} \right)} \right|} \right)\frac{{d\xi d\eta }}{{\| {{\eta ^{ - 1}} \circ \xi } \|_{{\mathbb{H}^n}}^Q}}} } \nonumber\\
    \le& c{\left( {{{\| \varphi  \|}_{{L^\infty }\left( {{B_r}} \right)}} + 1} \right)^q}\int_{{B_{2r}}} {\int_{{B_{2r}}} {H\left( {\xi ,\eta ,\left| {v \left( \xi  \right)} \right|} \right)\frac{{d\xi d\eta }}{{\| {{\eta ^{ - 1}} \circ \xi } \|_{{\mathbb{H}^n}}^Q}}} } \nonumber\\
    & + c\| {{\nabla _H}\varphi } \|_{{L^\infty }\left( {{B_r}} \right)}^p\int_{{B_{2r}}} {{{\left| {v \left( \xi  \right)} \right|}^p}\int_{{B_{4r}}} {\frac{{d\eta }}{{\| {{\eta ^{ - 1}} \circ \xi } \|_{{\mathbb{H}^n}}^{Q + (s - 1)p}}}} } \;d\xi \nonumber\\
    & + c\| {{\nabla _H}\varphi } \|_{{L^\infty }\left( {{B_r}} \right)}^q{\| a \|_{{L^\infty }}}\int_{{B_{2r}}} {{{\left| {v \left( \xi  \right)} \right|}^q}\int_{{B_{4r}}} {\frac{{d\eta }}{{\| {{\eta ^{ - 1}} \circ \xi } \|_{{\mathbb{H}^n}}^{Q + (t - 1)p}}}} } \;d\xi \nonumber\\
     \le& c{\left( {{{\| \varphi  \|}_{{L^\infty }\left( {{B_r}} \right)}} + 1} \right)^q}\rho \left( {v ;{B_{2r}}} \right) +  c\| {{\nabla _H}\varphi } \|_{{L^\infty }\left( {{B_r}} \right)}^p{r^{\left( {1 - s} \right)p}}\int_{{B_{2r}}} {{{\left| {v \left( \xi  \right)} \right|}^p}} \;d\xi \nonumber\\
     & + c\| {{\nabla _H}\varphi } \|_{{L^\infty }\left( {{B_r}} \right)}^q{\| a \|_{{L^\infty }}}{r^{\left( {1 - t} \right)q}}\int_{{B_{2r}}} {{{\left| {v \left( \xi  \right)} \right|}^q}} \;d\xi \nonumber\\
     <& \infty .
\end{align}
Thus, it follows $\rho \left( {v \varphi ;{\mathbb{H}^n}} \right) <\infty$ by combining \eqref{eq45}--\eqref{eq47} with \eqref{eq44}.
\end{proof}

Next, we prove a nonlocal Caccioppoli-type inequality. Define
\begin{equation}\label{eq48}
h\left( {\xi ,\eta ,\tau } \right): = \frac{{{\tau ^{p - 1}}}}{{\| {{\eta ^{ - 1}} \circ \xi } \|_{{\mathbb{H}^n}}^{sp}}} + a\left( {\xi ,\eta } \right)\frac{{{\tau ^{q - 1}}}}{{\| {{\eta ^{ - 1}} \circ \xi } \|_{{\mathbb{H}^n}}^{tq}}},\;\;\;\;\xi ,\eta  \in {\mathbb{H}^n}\;\hbox{and}\;\;\tau  > 0.
\end{equation}

The numerical inequality below, to be exploited frequently, is from \cite[Lemma 3.1]{DKP16}.

\begin{lemma}\label{Le26}
Let $p \ge 1$ and $a,b \ge 0$. Then we have
\[{a^p} - {b^p} \le p{a^{p - 1}}\left| {a - b} \right|\]
and
\[{a^p} - {b^p} \le \varepsilon {b^p} + c{\varepsilon ^{1 - p}}{\left| {a - b} \right|^p}\]
for any $\varepsilon  \in \left( {0,1} \right)$ and some $c=c(p)>0$.
\end{lemma}

\begin{lemma}\label{Le42}
Let $a:{\mathbb{H}^n}\times {\mathbb{H}^n} \to {\mathbb{R}}$ be symmetric and satisfy \eqref{eqy19} and \eqref{eqy110}, and $u\in {\mathcal{A}}(\Omega ) \cap L_{sp}^{q - 1}\left({{\mathbb{H}^n}}\right)$ be a weak solution to \eqref{eqy11}. Let ${B_{2r}} \equiv {B_{2r}}\left( {{\xi _0}} \right) \subset  \subset \Omega $ be a ball. Assume that \eqref{eqy112} holds or $u$ is bounded in ${B_{2r}}$. Then for any $\phi  \in C_0^\infty \left( {{B_r}} \right)$ with $0\le \phi \le 1$, we have
\begin{align}\label{eq49}
   & \int_{{B_r}} {\int_{{B_r}} {H\left( {\xi ,\eta ,\left| {{w_\pm }\left( \xi  \right) - {w_\pm }\left( \eta  \right)} \right|} \right)\left( {{\phi ^q}\left( \xi  \right) + {\phi ^q}\left( \eta  \right)} \right)\frac{{d\xi d\eta }}{{\| {{\eta ^{ - 1}} \circ \xi } \|_{{\mathbb{H}^n}}^Q}}} }\nonumber  \\
   \le&  c\int_{{B_r}} {\int_{{B_r}} {H\left( {\xi ,\eta ,\left| {\left( {\phi \left( \xi  \right) - \phi \left( \eta  \right)} \right)\left( {{w_\pm }\left( \xi  \right) + {w_\pm }\left( \eta  \right)} \right)} \right|} \right)\frac{{d\xi d\eta }}{{\| {{\eta ^{ - 1}} \circ \xi } \|_{{\mathbb{H}^n}}^Q}}} } \nonumber\\
   & + c\left( {\mathop {\sup }\limits_{\xi  \in \rm{supp}\,\phi } \int_{{\mathbb{H}^n}\backslash {B_r}} {h\left( {\xi ,\eta ,\left| {{w_\pm }\left( \eta  \right)} \right|} \right)\frac{{d\eta }}{{\| {{\eta ^{ - 1}} \circ \xi } \|_{{\mathbb{H}^n}}^Q}}} } \right)\int_{{B_r}} {{w_\pm }\left( \xi  \right){\phi ^q}\left( \xi  \right)d\xi }
\end{align}
for some $c\equiv c(Q,s,t,p,q,\Lambda )>0$, where ${w_\pm }:=(u-k)_ \pm$ with $k \ge 0$.
\end{lemma}

\begin{proof} We just consider the estimate for $w_+$, since the estimate for $w_-$ can be proved similarly. By Lemma \ref{Le41}, it follows that ${w_+}{\phi ^q} \in {\mathcal{A}}(\Omega )$ from $u \in {\mathcal{A}}(\Omega )$ and $\phi  \in C_0^\infty \left( {{B_r}} \right) \subset HW_0^{1,\infty} \left( {{B_r}} \right)$, so we can take the testing function $\varphi={w_+}{\phi ^q}$ in \eqref{eqy111}.
Then we have
\begin{align}\label{eq411}
  0
   = & \int_{{B_r}} {\int_{{B_r}} {\Bigg[\frac{{{J _p}\left( {u\left( \xi  \right) - u\left( \eta  \right)} \right)\left( {{w_ + }\left( \xi  \right){\phi^q}\left( \xi  \right) - {w_ + }\left( \eta  \right){\phi^q}\left( \eta  \right)} \right)}}{{\| {{\eta ^{ - 1}} \circ \xi } \|_{{\mathbb{H}^n}}^{Q + sp}}}} }  \nonumber\\
   &   + a\left( {\xi ,\eta } \right)\frac{{{J _q}\left( {u\left( \xi  \right) - u\left( \eta  \right)} \right)\left( {{w_ + }\left( \xi  \right){\phi^q}\left( \xi  \right) - {w_ + }\left( \eta  \right){\phi^q}\left( \eta  \right)} \right)}}{{\| {{\eta ^{ - 1}} \circ \xi } \|_{{\mathbb{H}^n}}^{Q + tq}}}\Bigg]\,d\xi d\eta  \nonumber\\
   &+2{\int _{{\mathbb{H}^n}\backslash {B_r}}}\int_{{B_r}} {\left[\frac{{{J _p}\left( {u\left( \xi  \right) - u\left( \eta  \right)} \right){w_ + }\left( \xi  \right){\phi^q}\left( \xi  \right)}}{{\| {{\eta ^{ - 1}} \circ \xi } \|_{{\mathbb{H}^n}}^{Q + sp}}} + a\left( {\xi ,\eta } \right)\frac{{{J _q}\left( {u\left( \xi  \right) - u\left( \eta  \right)} \right){w_ + }\left( \xi  \right){\phi^q}\left( \xi  \right)}}{{\| {{\eta ^{ - 1}} \circ \xi } \|_{{\mathbb{H}^n}}^{Q + tq}}}\right]\,d\xi d\eta }  \nonumber\\
   =: &J_1+J_2.
\end{align}

We first estimate $J_1$. Since $J_1$ is symmetry for $\xi$ and $\eta$, we may suppose without loss of generality that ${u\left( \xi  \right) \ge u\left( \eta  \right)}$. Then for $l \in \{p,q\}$, it yields
\begin{align*}
   & {J _l}\left( {u\left( \xi  \right) - u\left( \eta  \right)} \right)\left( {{w_+ }\left( \xi  \right){\phi^q}\left( \xi  \right) - {w_+ }\left( \eta  \right){\phi^q}\left( \eta  \right)} \right) \\
   =& {\left( {u\left( \xi  \right) - u\left( \eta  \right)} \right)^{l - 1}}\left( {{{\left( {u\left( \xi  \right) - k} \right)}_ + }{\phi^q}\left( \xi  \right) - {{\left( {u\left( \eta  \right) - k} \right)}_ + }{\phi^q}\left( \eta  \right)} \right)\\
    =& \left\{ \begin{array}{l}
{\left( {{w_+ }\left( \xi  \right) - {w_+ }\left( \eta  \right)} \right)^{l - 1}}\left( {{w_+ }\left( \xi  \right){\phi^q}\left( \xi  \right) - {w_+ }\left( \eta  \right){\phi^q}\left( \eta  \right)} \right),\;\;\;\, u\left( \xi  \right) \ge u\left( \eta  \right) \ge k\\
{\left( {u\left( \xi  \right) - u\left( \eta  \right)} \right)^{l - 1}}{w_+ }\left( \xi  \right){\phi^q}\left( \xi  \right),\;\;\;\;\;\;\;\;\;\;\;\;\;\;\;\;\;\;\;\;\;\;\;\;\;\;\;\;\;\;\;\;\; u\left( \xi  \right) \ge k \ge u\left( \eta  \right)\\
0,\;\;\;\;\;\;\;\;\;\;\;\;\;\;\;\;\;\;\;\;\;\;\;\;\;\;\;\;\;\;\;\;\;\;\;\;\;\;\;\;\;\;\;\;\;\;\;\;\;\;\;\;\;\;\;\;\;\;\;\;\;\;\;\;\;\;\;\;\;\;\;\;\, k \ge u\left( \xi  \right) \ge u\left( \eta  \right)
\end{array} \right.\\
\ge &{\left( {{w_+ }\left( \xi  \right) - {w_+ }\left( \eta  \right)} \right)^{l - 1}}\left( {{w_+ }\left( \xi  \right){\phi^q}\left( \xi  \right) - {w_+ }\left( \eta  \right){\phi^q}\left( \eta  \right)} \right)\\
 =& {J _l}\left( {{w_+ }\left( \xi  \right) - {w_+ }\left( \eta  \right)} \right)\left( {{w_+ }\left( \xi  \right){\phi^q}\left( \xi  \right) - {w_+ }\left( \eta  \right){\phi^q}\left( \eta  \right)} \right).
\end{align*}
Moverover,
\[{w_+ }\left( \xi  \right){\phi^q}\left( \xi  \right) - {w_+ }\left( \eta  \right){\phi^q}\left( \eta  \right) = \frac{{{w_+ }\left( \xi  \right) - {w_+ }\left( \eta  \right)}}{2}\left( {{\phi^q}\left( \xi  \right) + {\phi^q}\left( \eta  \right)} \right) + \frac{{{w_+ }\left( \xi  \right) + {w_+ }\left( \eta  \right)}}{2}\left( {{\phi^q}\left( \xi  \right) - {\phi^q}\left( \eta  \right)} \right),\]
which implies
\begin{align*}
   & {J _l}\left( {{w_+ }\left( \xi  \right) - {w_+ }\left( \eta  \right)} \right)\left( {{w_+ }\left( \xi  \right){\phi^q}\left( \xi  \right) - {w_+ }\left( \eta  \right){\phi^q}\left( \eta  \right)} \right) \\
\ge   &  {\left| {{w_+ }\left( \xi  \right) - {w_+ }\left( \eta  \right)} \right|^l}\frac{{{\phi^q}\left( \xi  \right) + {\phi^q}\left( \eta  \right)}}{2} - {\left| {{w_+ }\left( \xi  \right) - {w_+ }\left( \eta  \right)} \right|^{l - 1}}\frac{{{w_+ }\left( \xi  \right) + {w_+ }\left( \eta  \right)}}{2}\left| {{\phi^q}\left( \xi  \right) - {\phi^q}\left( \eta  \right)} \right|.
\end{align*}
Since
\begin{align*}
   &\left| {{\phi^q}\left( \xi  \right) - {\phi^q}\left( \eta  \right)} \right| \le q\left( {{\phi^{q - 1}}\left( \xi  \right) + {\phi^{q - 1}}\left( \eta  \right)} \right)\left| {\phi\left( \xi  \right) - \phi\left( \eta  \right)} \right| \\
  \le &  c\left( q \right){\left( {{\phi^q}\left( \xi  \right) +{\phi^q}\left( \eta  \right)} \right)^{\frac{{q - 1}}{q}}}\left| {\phi \left( \xi  \right) - \phi\left( \eta  \right)} \right|
\end{align*}
from Lemma \ref{Le26}, we use Young's inequality, $0\le \phi\le 1$ and ${\frac{{q - 1}}{q}}>0$ to deduce that
\begin{align*}
   & {\left| {{w_+ }\left( \xi  \right) - {w_+ }\left( \eta  \right)} \right|^{l - 1}}\frac{{{w_+ }\left( \xi  \right) + {w_+ }\left( \eta  \right)}}{2}\left| {{\phi^q}\left( \xi  \right) - {\phi^q}\left( \eta  \right)} \right| \\
   \le&  c\left( q \right){\left| {{w_+ }\left( \xi  \right) - {w_+ }\left( \eta  \right)} \right|^{l - 1}}\left( {{w_+ }\left( \xi  \right) + {w_+ }\left( \eta  \right)} \right){\left( {{\phi^q}\left( \xi  \right) + {\phi^q}\left( \eta  \right)} \right)^{\frac{{q - 1}}{q}}}\left| {\phi\left( \xi  \right) - \phi\left( \eta  \right)} \right|\\
   =& c\left( q \right){\left| {{w_+ }\left( \xi  \right) - {w_+ }\left( \eta  \right)} \right|^{l - 1}}\left( {{w_+ }\left( \xi  \right) + {w_+ }\left( \eta  \right)} \right){\left( {{\phi^q}\left( \xi  \right) + {\phi^q}\left( \eta  \right)} \right)^{\frac{{l - 1}}{l} + \frac{{q - l}}{{ql}}}}\left| {\phi\left( \xi  \right) - \phi\left( \eta  \right)} \right|\\
    \le& \varepsilon {\left| {{w_+ }\left( \xi  \right) - {w_+ }\left( \eta  \right)} \right|^l}\left( {{\phi^q}\left( \xi  \right) + {\phi^q}\left( \eta  \right)} \right) + c\left( {\varepsilon ,q} \right){\left( {{\phi^q}\left( \xi  \right) + {\phi^q}\left( \eta  \right)} \right)^{\frac{{q - l}}{q}}}{\left| {\phi\left( \xi  \right) -\phi\left( \eta  \right)} \right|^l}{\left( {{w_+ }\left( \xi  \right) + {w_+ }\left( \eta  \right)} \right)^l}\\
     \le& \varepsilon {\left| {{w_+ }\left( \xi  \right) - {w_+ }\left( \eta  \right)} \right|^l}\left( {{\phi^q}\left( \xi  \right) + {\phi^q}\left( \eta  \right)} \right) + c\left( {\varepsilon ,q} \right){\left| {\phi\left( \xi  \right) -\phi\left( \eta  \right)} \right|^l}{\left( {{w_+ }\left( \xi  \right) + {w_+ }\left( \eta  \right)} \right)^l}.
\end{align*}
Then, by choosing $\varepsilon$ small enough, we have
\begin{align*}
   & {J _l}\left( {{w_+ }\left( \xi  \right) - {w_+ }\left( \eta  \right)} \right)\left( {{w_+ }\left( \xi  \right){\phi^q}\left( \xi  \right) - {w_+ }\left( \eta  \right){\phi^q}\left( \eta  \right)} \right) \\
  \ge &  {\left| {{w_+ }\left( \xi  \right) - {w_+ }\left( \eta  \right)} \right|^l}\frac{{{\phi^q}\left( \xi  \right) + {\phi^q}\left( \eta  \right)}}{4} - c{\left| {\phi( \xi) -\phi\left( \eta  \right)} \right|^l}{\left( {{w_+ }\left( \xi  \right) + {w_+ }\left( \eta  \right)} \right)^l}.
\end{align*}
Thus, we get
\begin{align}\label{eq412}
   {J_1} \ge&\int_{{B_r}} {\int_{{B_r}} {\Bigg[ {\frac{{{{\left| {{w_ + }\left( \xi  \right) - {w_ + }\left( \eta  \right)} \right|}^p}(\phi^q(\xi) + \phi^q( \eta))/4 - c{{\left| {\phi( \xi) - \phi( \eta)} \right|}^p}{{\left( {{w_ + }\left( \xi  \right) + {w_ + }\left( \eta  \right)} \right)}^p}}}{{\| {{\eta ^{ - 1}} \circ \xi } \|_{{\mathbb{H}^n}}^{Q + sp}}}} } } \nonumber \\
   &    + a\left( {\xi ,\eta } \right)\frac{{{{\left| {{w_ + }\left( \xi  \right) - {w_ + }\left( \eta  \right)} \right|}^q}(\phi^q(\xi) +\phi^q(\eta))/4 - c{{\left| {\phi( \xi) -\phi( \eta)} \right|}^q}{{\left( {{w_ + }\left( \xi  \right) + {w_ + }\left( \eta  \right)} \right)}^q}}}{{\| {{\eta ^{ - 1}} \circ \xi } \|_{{\mathbb{H}^n}}^{Q + tq}}}\Bigg]\,d\xi d\eta \nonumber \\
    \ge&  {\int_{{B_r}} {H\left( {\xi ,\eta ,\left| {{w_+ }\left( \xi  \right) - {w_+ }\left( \eta  \right)} \right|} \right)\left( {\phi^q\left( \xi  \right) + {\phi^q}\left( \eta  \right)} \right)\frac{{d\xi d\eta }}{{\| {{\eta ^{ - 1}} \circ \xi } \|_{{\mathbb{H}^n}}^Q}}} } \nonumber\\
    &  - c\int_{{B_r}} {\int_{{B_r}} {H\left( {\xi ,\eta ,\left| {\phi( \xi) -\phi( \eta)} \right|\left( {{w_+ }\left( \xi  \right) + {w_+ }\left( \eta  \right)} \right)} \right)\frac{{d\xi d\eta }}{{\| {{\eta ^{ - 1}} \circ \xi } \|_{{\mathbb{H}^n}}^Q}}} } .
\end{align}

Now we estimate $J_2$. Note that
\begin{equation}\label{eq413}
{J _l}\left( {u\left( \xi  \right) - u\left( \eta  \right)} \right){w_+ }\left( \xi  \right) \ge  - w_+ ^{l - 1}\left( \eta  \right){w_+ }\left( \xi  \right).
\end{equation}
In fact, when $u\left( \xi  \right) \ge u\left( \eta  \right)$, it easy to see that the inequality \eqref{eq413} holds. When $u\left( \xi  \right) < u\left( \eta  \right)$ and $u\left( \xi  \right)\le k$, ${w_+ }\left( \xi  \right)=0$, the inequality \eqref{eq413} also holds. When $k<u\left( \xi  \right) < u\left( \eta  \right)$,
\begin{align*}
  {J _l}\left( {u\left( \xi  \right) - u\left( \eta  \right)} \right){w_+ }\left( \xi  \right)  
 &=  - {\left| {u\left( \xi  \right) - u\left( \eta  \right)} \right|^{l - 1}}{w_+ }\left( \xi  \right) \\
 &=  - {\left| {{w_+ }\left( \xi  \right) - {w_+ }\left( \eta  \right)} \right|^{l - 1}}{w_+ }\left( \xi  \right)\\
 & \ge  - w_+ ^{l - 1}\left( \eta  \right){w_+ }\left( \xi  \right).
\end{align*}
Thus, we apply \eqref{eq413} and \eqref{eq48} to get
\begin{align}\label{eq414}
  J_2 =& 2{\int _{{\mathbb{H}^n}\backslash {B_r}}}\int_{{B_r}} {\left[\frac{{{J _p}\left( {u\left( \xi  \right) - u\left( \eta  \right)} \right){w_ + }\left( \xi  \right){\phi^q}\left( \xi  \right)}}{{\| {{\eta ^{ - 1}} \circ \xi } \|_{{\mathbb{H}^n}}^{Q + sp}}} + a\left( {\xi ,\eta } \right)\frac{{{J _q}\left( {u\left( \xi  \right) - u\left( \eta  \right)} \right){w_ + }\left( \xi  \right){\phi^q}\left( \xi  \right)}}{{\| {{\eta ^{ - 1}} \circ \xi } \|_{{\mathbb{H}^n}}^{Q + tq}}}\right]\,d\xi d\eta } \nonumber\\
   \ge& - c{\int _{{\mathbb{H}^n}\backslash {B_r}}}\int_{{B_r}} {\left[ {\frac{{w_+ ^{p - 1}\left( \eta  \right){w_+ }\left( \xi  \right){\phi^q}\left( \xi  \right)}}{{\| {{\eta ^{ - 1}} \circ \xi } \|_{{\mathbb{H}^n}}^{sp}}} + a\left( {\xi ,\eta } \right)\frac{{w_+ ^{q - 1}\left( \eta  \right){w_+ }\left( \xi  \right){\phi^q}\left( \xi  \right)}}{{\| {{\eta ^{ - 1}} \circ \xi } \|_{{\mathbb{H}^n}}^{tq}}}} \right]\frac{{d\xi d\eta }}{{\| {{\eta ^{ - 1}} \circ \xi } \|_{{\mathbb{H}^n}}^Q}}} \nonumber \\
    = & - c{\int _{{\mathbb{H}^n}\backslash {B_r}}}\int_{{B_r}} {h\left( {\xi ,\eta ,{w_+ }\left( \eta  \right)} \right){w_+ }\left( \xi  \right){\phi^q}\left( \xi  \right)\frac{d\xi d\eta }{{\| {{\eta ^{ - 1}} \circ \xi } \|_{{\mathbb{H}^n}}^Q}}}  \nonumber\\
     \ge&  - c\left( {\mathop {\sup }\limits_{\xi  \in {\rm{supp}}\;\phi} \int_{{\mathbb{H}^n}\backslash {B_r}} {h\left( {\xi ,\eta ,{w_+ }\left( \eta  \right)} \right)\frac{d\eta }{{\| {{\eta ^{ - 1}} \circ \xi } \|_{{\mathbb{H}^n}}^Q}}} } \right)\int_{{B_r}} {{w_\pm }\left( \xi  \right){\phi^q}\left( \xi  \right)\,d\xi }.
\end{align}
Combining \eqref{eq411}, \eqref{eq412} with \eqref{eq414}, we get \eqref{eq49}.
\end{proof}
The following standard iteration lemma can be found in \cite[Lemma 7.1]{G03}.

\begin{lemma}\label{Le27}
Let $\left\{ {{y_i}} \right\}_{i = 0}^\infty $ be a sequence of nonnegative numbers satisfying
\[{y_{i + 1}} \le {b_1}b_2^iy_i^{1 + \beta },\;i = 0,1,2, \cdots \]
for some constants $b_1,\;\beta>0$ and $b_2>1$. If
\[{y_0} \le b_1^{ - \frac{1}{\beta }}b_2^{ - \frac{1}{{{\beta ^2}}}},\]
then ${y_i} \to 0$ as $i \to \infty$.
\end{lemma}

We end this section by providing the proof of boundedness. Lemmas \ref{Le24} and \ref{Le42} play the vital roles in the process.

\medskip

\noindent \textbf{Proof of Theorem \ref{Th11}.} For convenience, we denote
\[{H_0}\left( \tau  \right) = {\tau ^p} + {\| a \|_{{L^\infty }}}{\tau ^q},\;\;\;\;\tau  \ge 0.\]
Let ${B_{r}} \equiv {B_{r}}\left( {{\xi _0}} \right) \subset  \subset \Omega $ be a fixed ball with $r \le 1$. For $i=0,1,2,\cdots$ and $k_0 >0$, we write
\[{r_i}: = \frac{r}{2}\left( {1 + {2^{ - i}}} \right),\quad {\sigma _i}: = \;\frac{{{r_{i - 1}} + {r_i}}}{2}, \quad {k_i}: = 2{k_0}\left( {1 - {2^{ - i - 1}}} \right)\]
and 
\[{y_i}: = \int_{{A^ + }\left( {{k_i},{r_i}} \right)} {{H_0}\left( {\left( {u\left( \xi  \right) - {k_i}} \right){_ + }} \right)\;d\xi } .\]
In addition, we denote
\[{A^ + }\left( {{k_i},{r_i}} \right): = \left\{ {\xi  \in {B_{{r_i}}}:u\left( \xi  \right) \ge {k_i}} \right\}.\]
Then
\begin{equation}\label{eqb41}
 {A^ + }\left( {{k_i},{r_i}} \right) \subset {A^ + }\left( {{k_{i - 1}},{r_i}} \right) \subset {A^ + }\left( {{k_{i - 1}},{r_{i - 1}}} \right).
\end{equation}
Moreover, for $\xi \in {A^ + }\left( {{k_i},{r_i}} \right)$, we have
\[{\left( {u\left( \xi  \right) - k_{i-1}} \right)_ + } = u\left( \xi  \right) - k_{i-1} \ge k_i - k_{i-1}=2^{-i}k_0\]
and
\[{\left( {u\left( \xi  \right) - k_{i-1}} \right)_ + } = u\left( \xi  \right) - k_{i-1} \ge u\left( \xi  \right) - k_i = {\left( {u\left( \xi  \right) - k_i} \right)_ + }.\]
Thus, it deduces
\begin{equation}\label{eq416}
\left| {{A^ + }\left( {k_i,r_i } \right)} \right|   \le \int_{{A^ + }\left( {k_i,r_i } \right)} {\frac{{\left( {u\left( \xi  \right) - k_{i-1}} \right)_ + ^p}}{{{{\left( {k_i - k_{i-1}} \right)}^p}}}d\xi }  \le {k_0^{-p}}2^{ip}\int_{{A^ + }\left( {k_i,r_i } \right)} {{H_0}\left( {{{\left( {u\left( \xi  \right) - k_{i-1}} \right)}_ + }} \right)d\xi }\le{k_0^{-p}}2^{ip}y_{i-1}
\end{equation}
and
\begin{align}\label{eq417}
  \int_{{B_{r_{i-1}} }} {{{\left( {u\left( \xi  \right) - k_i} \right)}_ + }\;d\xi } &\le \int_{{B_{r_{i-1}} }} {{{\left( {u\left( \xi  \right) - k_{i-1}} \right)}_ + }\;d\xi }  \nonumber\\
  & \le \int_{{B_{r_{i-1}} }} {{{\left( {u\left( \xi  \right) - k_{i-1}} \right)}_ + }{{\left( {\frac{{{{\left( {u\left( \xi  \right) - k_{i-1}} \right)}_ + }}}{{k_i - k_{i-1}}}} \right)}^{p - 1}}\;d\xi }  \nonumber\\
   &  \le {{{k_0^{1-p}}{2^{i(p - 1)}}}}\int_{{B_{r_{i-1}} }} {{H_0}\left( {{{\left( {u\left( \xi  \right) - k_{i-1}} \right)}_ + }} \right)\;d\xi }  \nonumber\\
   &={{{k_0^{1-p}}{2^{i(p - 1)}}}}y_{i-1}.
\end{align}
Then we choose a cut-off function $\phi \in C_0^\infty \left( {{B_{\frac{{\sigma _i + {r_{i - 1}}}}{2}}}} \right)$ satisfying $0 \le \phi\le 1,\;\phi\equiv 1 $ in $B_{\sigma _i}$ and $\left| {{\nabla _H}\phi } \right| \le \frac{c}{{r_{i-1}-\sigma _i  }}=\frac{c}{r}2^i$. Denoting the tail by
\[T\left( {v;r_{i-1}} \right): = \int_{{\mathbb{H}^n}\backslash {B_{r_{i-1}}}} {\left( {\frac{{{{\left| {v \left( \xi  \right)} \right|}^{p - 1}}}}{{\| {\xi _0^{ - 1} \circ \xi } \|_{{\mathbb{H}^n}}^{Q + sp}}} + {{\| a \|}_{{L^\infty }}}\frac{{{{\left| {v \left( \xi  \right)} \right|}^{q - 1}}}}{{\| {\xi _0^{ - 1} \circ \xi } \|_{{\mathbb{H}^n}}^{Q + tq}}}} \right)\;d\xi .} \]

We use Lemma \ref{Le24} with $f \equiv {\left( {u - k} \right)_ + }$ and \eqref{eq416} to get
\begin{align}\label{eq415}
  {y_i} &= \int_{{A^ + }\left( {{k_i},{r_i}} \right)} {{H_0}\left( {{{\left( {u\left( \xi  \right) - {k_i}} \right)}_ + }} \right)\;d\xi } \nonumber \\
  & \le \int_{{B_{{r_i}}}} {{H_0}\left( {{{\left( {u\left( \xi  \right) - {k_i}} \right)}_ + }} \right)\;d\xi } \nonumber \\
  &= c{r_i^Q}\fint_{{B_{{r_i}}}} {{H_0}\left( {{{\left( {u\left( \xi  \right) - {k_i}} \right)}_ + }} \right)\;d\xi }\nonumber\\
 & \le cr_i^{Q + sp } \fint_{{B_{{{r _i}}}}} {\left( {{{\left| {\frac{{{{\left( {u\left( \xi  \right) - {k_i}} \right)}_ + }}}{{{r_i^s}}}} \right|}^p} + {a_0}{{\left| {\frac{{{{\left( {u\left( \xi  \right) - {k_i}} \right)}_ + }}}{{{r_i^t}}}} \right|}^q}} \right)\;d\xi }  \nonumber\\
  & \le c{a_0}r_i^{Q + sp - tq}D_1^{\frac{q}{p}}({\sigma _i},{r_i}){\left( {\fint_{{B_{{\sigma _i}}}} {\int_{{B_{{{\sigma _i}}}}} {\frac{{{{\left| {{{\left( {u\left( \xi  \right) - {k_i}} \right)}_ + } - {{\left( {u\left( \eta  \right) - {k_i}} \right)}_ + }} \right|}^p}}}{{\| {{\eta ^{ - 1}} \circ \xi } \|_{{\mathbb{H}^n}}^{Q + sp}}}\;d\xi } d\eta } } \right)^{\frac{q}{p}}}\nonumber\\
   & \quad+ cr_i^{Q - sp}{D_1}({\sigma _i},r_i){\left( {{A^ + }\left( {{k_i},{r_i}} \right)} \right)^{\frac{{sp}}{Q}}}\fint_{{B_{{\sigma _i}}}} {\int_{{B_{{\sigma _i}}}} {\frac{{{{\left| {{{\left( {u\left( \xi  \right) - {k_i}} \right)}_ + } - {{\left( {u\left( \eta  \right) - {k_i}} \right)}_ + }} \right|}^p}}}{{\| {{\eta ^{ - 1}} \circ \xi } \|_{{\mathbb{H}^n}}^{Q + sp}}}\;d\xi } d\eta } \nonumber\\
   &\quad + cr_i^{Q + sp}\left[ {{{\left( {\frac{{{A^ + }\left( {{k_i},{r_i}} \right)}}{{\left| {{B_{{r_i}}}} \right|}}} \right)}^{p - 1}} + D_2^{\frac{q}{p}}({\sigma _i},r_i)\left( {{{\left( {\frac{{{A^ + }\left( {{k_i},{r_i}} \right)}}{{\left| {{B_{{r_i}}}} \right|}}} \right)}^{\frac{{sp}}{Q}}} + {{\left( {\frac{{{A^ + }\left( {{k_i},{\sigma _i}} \right)}}{{\left| {{B_{{{\sigma _i}}}}} \right|}}} \right)}^{\frac{{q - p}}{p}}}} \right)} \right] \nonumber\\
   & \qquad \cdot \fint_{{B_{{{\sigma _i}}}}} {\left( {{{\left| {\frac{{{{\left( {u\left( \xi  \right) - {k_i}} \right)}_ + }}}{{{r_i^s}}}} \right|}^p} + {a_0}{{\left| {\frac{{{{\left( {u\left( \xi  \right) - {k_i}} \right)}_ + }}}{{{r_i^t}}}} \right|}^q}} \right)\;d\xi }  \nonumber\\
  & \le c{a_0}r_i^{Q + sp - tq}D_1^{\frac{q}{p}}({\sigma _i},{r_i}){\left( {\fint_{{B_{{{\sigma _i}}}}} {\int_{{B_{{{\sigma _i}}}}} {\frac{{{{\left| {{{\left( {u\left( \xi  \right) - {k_i}} \right)}_ + } - {{\left( {u\left( \eta  \right) - {k_i}} \right)}_ + }} \right|}^p}}}{{\|  {{\eta ^{ - 1}} \circ \xi } \|_{{\mathbb{H}^n}}^{Q + sp}}}\;d\xi } d\eta } } \right)^{\frac{q}{p}}}\nonumber\\
   &\quad + ck_0^{ - \frac{{s{p^2}}}{Q}}r_i^{Q - sp}{2^{i\frac{{s{p^2}}}{Q}}}{D_1}({\sigma _i},r_i)y_{i - 1}^{\frac{{sp}}{Q}}\fint_{{B_{{\sigma _i}}}} {\int_{{B_{{\sigma _i}}}} {\frac{{{{\left| {{{\left( {u\left( \xi  \right) - {k_i}} \right)}_ + } - {{\left( {u\left( \eta  \right) - {k_i}} \right)}_ + }} \right|}^p}}}{{\| {{\eta ^{ - 1}} \circ \xi } \|_{{\mathbb{H}^n}}^{Q + sp}}}\;d\xi } d\eta } \nonumber\\
  & \quad + cr_i^{Q + sp-tq}\left[ {{{\left( {\frac{{k_0^{ - p}{2^{ip}}{y_{i - 1}}}}{{\left| {{B_{{r_i}}}} \right|}}} \right)}^{p - 1}} + D_2^{\frac{q}{p}}({\sigma _i},r_i)\left( {{{\left( {\frac{{k_0^{ - p}{2^{ip}}{y_{i - 1}}}}{{\left| {{B_{{r_i}}}} \right|}}} \right)}^{\frac{{sp}}{Q}}} + {{\left( {\frac{{k_0^{ - p}{2^{ip}}{y_{i - 1}}}}{{\left| {{B_{{\sigma _i}}}} \right|}}} \right)}^{\frac{{q - p}}{p}}}} \right)} \right] \nonumber\\
  &\qquad \cdot \fint_{{B_{{\sigma _i}}}} {{H_0}\left( {{{\left( {u\left( \xi  \right) - {k_i}} \right)}_ + }} \right)\;d\xi } .
\end{align}
Since we have that from Lemma \ref{Le42} and \eqref{eq417}
\begin{align*}
   & \fint_{{B_{{\sigma _i}}}} {\int_{{B_{{\sigma _i}}}} {\frac{{{{\left| {{{\left( {u\left( \xi  \right) - {k_i}} \right)}_ + } - {{\left( {u\left( \eta  \right) - {k_i}} \right)}_ + }} \right|}^p}}}{{\| {{\eta ^{ - 1}} \circ \xi } \|_{{\mathbb{H}^n}}^{Q + sp}}}\;d\xi } d\eta }  \\
 \le&  \fint_{{B_{{\sigma _i}}}} {\int_{{B_{{\sigma _i}}}} {H\left( {\xi ,\eta ,\left| {{{\left( {u\left( \xi  \right) - {k_i}} \right)}_ + } - {{\left( {u\left( \eta  \right) - {k_i}} \right)}_ + }} \right|} \right)\frac{{d\xi d\eta }}{{\| {{\eta ^{ - 1}} \circ \xi } \|_{{\mathbb{H}^n}}^Q}}} }  \\
 \le &c{r^{ - p}}{2^{ip}}\fint_{{B_{{r_{i - 1}}}}} {\left( {u\left( \xi  \right) - {k_{i }}} \right)_ + ^p\int_{{B_{{r_{i - 1}}}}} {\frac{{d\xi d\eta }}{{\| {{\eta ^{ - 1}} \circ \xi } \|_{{\mathbb{H}^n}}^{Q + (s - 1)p}}}} } \\
 & + c{\| a \|_{{L^\infty }}}{r^{ - q}}{2^{iq}}\fint_{{B_{{r_{i - 1}}}}} {\left( {u\left( \xi  \right) - {k_{i }}} \right)_ + ^q\int_{{B_{{r_{i - 1}}}}} {\frac{{d\xi d\eta }}{{\| {{\eta ^{ - 1}} \circ \xi } \|_{{\mathbb{H}^n}}^{Q + (t - 1)q}}}} } \\
 & + c\left( {\mathop {\sup }\limits_{\xi  \in {\rm{supp}}\;\phi } \int_{{\mathbb{H}^n}\backslash {B_{{r_{i - 1}}}}} {\left( {\frac{{\left( {u\left( \eta  \right) - k_{i }} \right)_ + ^{p - 1}}}{{\| {{\eta ^{ - 1}} \circ \xi } \|_{{\mathbb{H}^n}}^{Q + sp}}} + {{\| a \|}_{{L^\infty }}}\frac{{\left( {u\left( \eta  \right) - k_{i }} \right)_ + ^{q - 1}}}{{\| {{\eta ^{ - 1}} \circ \xi } \|_{{\mathbb{H}^n}}^{Q + tq}}}} \right)\;d\eta } } \right)\fint_{{B_{{r_{i - 1}}}}} {{{\left( {u\left( \xi  \right) - k_{i }} \right)}_ + }\;d\xi } \\
  \le& c{r^{ - p}}{2^{ip}}r_{i - 1}^{\left( {1 - s} \right)p}\fint_{{B_{{r_{i - 1}}}}} {\left( {u\left( \xi  \right) - {k_{i }}} \right)_ + ^p\;d\xi}  + c{\| a \|_{{L^\infty }}}{r^{ - q}}{2^{iq}}r_{i - 1}^{\left( {1 - t} \right)q}\fint_{{B_{{r_{i - 1}}}}} {\left( {u\left( \xi  \right) - {k_{i }}} \right)_ + ^q\;d\xi}+ c{\left( {\frac{{{r_{i - 1}} + {\sigma _i}}}{{{r_{i - 1}} - {\sigma _i}}}} \right)^{Q + tq}} \\
  &\cdot\left( {\int_{{\mathbb{H}^n}\backslash {B_{{r_{i - 1}}}}} {\left( {\frac{{\left( {u\left( \eta  \right) - k_{i }} \right)_ + ^{p - 1}}}{{\| {{\eta ^{ - 1}} \circ {\xi _0}} \|_{{\mathbb{H}^n}}^{Q + sp}}} + {{\| a \|}_{{L^\infty }}}\frac{{\left( {u\left( \eta  \right) - k_{i }} \right)_ + ^{q - 1}}}{{\| {{\eta ^{ - 1}} \circ {\xi _0}} \|_{{\mathbb{H}^n}}^{Q + tq}}}} \right)\;d\eta } } \right)\fint_{{B_{{r_{i - 1}}}}} {{{\left( {u\left( \xi  \right) - k_{i }} \right)}_ + }\;d\xi } \\
  \le& c{r^{ - q}}{2^{iq}}r_{i - 1}^{\left( {1 - t} \right)p}\fint_{{B_{{r_{i - 1}}}}} {{H_0}\left( {{{\left( {u\left( \xi  \right) - {k_i}} \right)}_ + }} \right)\;d\xi }  + c{2^{i\left( {Q + tq} \right)}}T\left( {{{\left( {u - {k_i}} \right)}_ + };{r_{i - 1}}} \right)\fint_{{B_{{r_{i - 1}}}}} {{{\left( {u\left( \xi  \right) - k_i} \right)}_ + }\;d\xi }\\
 \le &c{2^{i\left( {Q + q + p-1 } \right)}}{y_{i - 1}},
\end{align*}
where we used the fact that
\[T\left( {{{\left( {u - {k_i}} \right)}_ + };{r_{i - 1}}} \right) \le T\left( {u;\frac{r}{2}} \right) < \infty ,\]
and
\[\frac{{\| {{\eta ^{ - 1}} \circ {\xi _0}} \|_{{\mathbb{H}^n}}^{}}}{{\| {{\eta ^{ - 1}} \circ \xi } \|_{{\mathbb{H}^n}}^{}}} \le 1 + \frac{{\| {\xi _0^{ - 1} \circ \xi } \|_{{\mathbb{H}^n}}^{}}}{{\| {{\eta ^{ - 1}} \circ \xi } \|_{{\mathbb{H}^n}}^{}}} \le 1 + \frac{{{r_{i - 1}} + {\sigma _i}}}{{{r_{i - 1}} - {\sigma _i}}} \le 2\frac{{{r_{i - 1}} + {\sigma _i}}}{{{r_{i - 1}} - {\sigma _i}}} \le c{2^i}\]
for ${\xi  \in {\rm{supp}}\;\phi }$ and $\eta \in {{\mathbb{H}^n}\backslash {B_{r_{i - 1}} }}$,
and note that
\[{D_1}({\sigma _i},{r_i}) \le c{2^{i\left( {Q + p} \right)}}, \quad {D_2}({\sigma _i},r_i) \le c{2^{i\left( {Q + p} \right)}},\]
it follows from \eqref{eq415} that
\begin{align}\label{eq418}
{y_i} \le & c{a_0}{2^{i\left[ {\frac{{q\left( {Q + p} \right)}}{p} + \frac{{q\left( {Q + q + p-1} \right)}}{p}} \right]}}y_{i - 1}^{\frac{q}{p}} + c{2^{i\left( {\frac{{{p^2}}}{Q} + Q + p + \left( {Q + q + p-1} \right)} \right)}}y_{i - 1}^{\frac{{sp}}{Q} + 1}\nonumber\\
& + c{2^{ip(p - 1)}}y_{i - 1}^p + c{2^{i\left[ {\frac{{q\left( {Q + p} \right)}}{p} + \frac{p^2}{Q}} \right]}}y_{i - 1}^{\frac{{sp}}{Q} + 1} + c{2^{i\left[ {\frac{{q\left( {Q + p} \right)}}{p} +q- p} \right]}}y_{i - 1}^{\frac{q}{p}}\nonumber\\
  \le& c{2^{i\left[ {\frac{{q\left( {Q + p} \right)}}{p} + \frac{{q\left( {Q + q + p} \right)}}{p}} \right]}}y_{i - 1}^{\frac{q}{p}} + c{2^{i\left( {\frac{{{p^2}}}{Q} + Q + p + \frac{{q\left( {Q + q + p} \right)}}{p}} \right)}}y_{i - 1}^{\frac{{sp}}{Q} + 1} + c{2^{ip(p - 1)}}y_{i - 1}^p.
\end{align}

Since ${H_0}\left( u \right) \in {L^1}\left( \Omega  \right)$ from the assumption \eqref{eqy112}, we get that
\[{y_0} = \int_{{A^ + }\left( {{k_0},r} \right)} {{H_0}\left( {\left( {u\left( \xi  \right) - {k_0}} \right){_ + }} \right)\;d\xi }  \to 0\quad \text{as } k_0 \to \infty .\]
First, we consider $k_0 >1$ so large that
\[{y_i} \le {y_{i - 1}} \le  \cdots  \le {y_0} \le 1,\;\;i = 1,2, \cdots .\]
Then, we have from \eqref{eq418} that
$${y_i} \le c{2^{\theta i}}y_{i - 1}^\beta ,$$
 where
\[\theta  = 2\left(\frac{{\left( {Q + p + q} \right)q}}{p} + {p^2}\right),\;\;\;\;\;\beta  = \min \left\{ {\frac{{q }}{p}-1,\frac{{sp}}{Q},p - 1} \right\}.\]
Finally, we can choose $k_0$ so large that
\[{y_0} \le {{\tilde c}^{ - \frac{1}{\beta }}}{2^{ - \frac{\theta }{{{\beta ^2}}}}}\]
holds. Then Lemma \ref{Le27} implies
\[{y_\infty } = \int_{{A^ + }\left( {2{k_0},\frac{r}{2}} \right)} {{H_0}\left( {\left( {u\left( \xi  \right) - 2{k_0}} \right){_ + }} \right)\;d\xi }  = 0,\]
which means that $u \le 2{k_0}$ a.e. in ${B_{\frac{r}{2}}}$.

Applying the same argument to $-u$, we consequently obtain $u \in {L^\infty }( {{B_{\frac{r}{2}}}} )$.

\section{H\"{o}lder continuity}
\label{Section 5}

We are going to demonstrate the H\"{o}lder regularity of weak solutions to Eq. \eqref{eqy11} in the last section. First, the second important tool, logarithmic estimate, is established as follows. Throughout this part, we fix any subdomain $\Omega ' \subset  \subset \Omega $.

\begin{lemma}[Logarithmic inequality]\label{Le51}
Let $s,t,p,q$ satisfy \eqref{eqy18} and $a$
fulfill \eqref{eqy19}, \eqref{eqy110} with \eqref{eqy113}. Let also $u\in \mathcal{A}(\Omega)$ be a weak solution of \eqref{eqy11} such that $u\in L^{\infty}(\Omega')$ and $u \ge 0$ in ${B_R}: = {B_R}\left( {{\xi _0}} \right) \subset \Omega ' \subset  \subset \Omega $ with $R\le 1$. Then there is a constant $c>1$ such that, for any $0<r\le \frac{R}{2}$ and $d>0$,
\begin{align*}
   & \int_{{B_r}} {\int_{{B_r}} {\left| {\log \frac{{u\left( \xi  \right) + d}}{{u\left( \eta  \right) + d}}} \right|} } \frac{{d\xi d\eta }}{{\| {{\eta ^{ - 1}} \circ \xi } \|_{{\mathbb{H}^n}}^Q}} \\
\le   &  c{K^2}\left( {{r^Q} + \frac{{{r^{Q + sp}}}}{{{d^{p - 1}}}}\int_{{\mathbb{H}^n}\backslash {B_R}} {\frac{{u_ - ^{p - 1}\left( \eta  \right) + u_ - ^{q - 1}\left( \eta  \right)}}{{\| {{\eta ^{ - 1}} \circ \xi_0 } \|_{{\mathbb{H}^n}}^{Q + sp}}}\,d\eta }  + \frac{{{r^{Q + tq}}}}{{{r^{q - 1}}}}\int_{{\mathbb{H}^n}\backslash {B_R}} {\frac{{u_ - ^{q - 1}\left( \eta  \right)}}{{\| {{\eta ^{ - 1}} \circ \xi_0 } \|_{{\mathbb{H}^n}}^{Q + tq}}}\,d\eta } } \right)
\end{align*}
holds true. Here $K:=1+d^{q-p}+\| u \|_{{L^\infty }\left( {\Omega '} \right)}^{q - p}$ and $c$ depends on $n,p,q,s,t$ and $\alpha, {\left[ a \right]_\alpha },\| a \|_{L^\infty }$.
\end{lemma}

\begin{proof} Let us give some notations as below,
\[{H_\rho }\left( {\xi ,\eta ,\tau } \right) = \frac{{{\tau ^p}}}{{{\rho ^{sp}}}} + a\left( {\xi ,\eta } \right)\frac{{{\tau ^q}}}{{{\rho ^{tq}}}},\;\;\;{h_\rho}\left( {\xi ,\eta ,\tau } \right) = \frac{{{\tau ^{p - 1}}}}{{{\rho^{sp}}}} + a\left( {\xi ,\eta } \right)\frac{{{\tau ^{q - 1}}}}{{{\rho^{tq}}}}\]
and
\[{G_\rho }\left( \tau  \right) = \frac{{{\tau ^p}}}{{{\rho ^{sp}}}} + a_\rho ^ + \frac{{{\tau ^q}}}{{{\rho ^{tq}}}},\;\;\;{g_\rho }\left( \tau  \right) = \frac{{{\tau ^{p - 1}}}}{{{\rho ^{sp}}}} + a_\rho ^ + \frac{{{\tau ^{q - 1}}}}{{{\rho ^{tq}}}}\]
with $a_\rho ^ + : = \mathop {\sup }\limits_{{B_\rho } \times {B_\rho }} a\left( { \cdot , \cdot } \right)$ and $\tau \ge 0$.

Consider a cut-off function $\phi\in C_0^\infty \left( {{B_{\frac{{3r}}{2}}}\left( {{\xi _0}} \right)} \right)$ satisfying
\begin{center}
$0 \le \phi\le 1$, \quad $\phi\equiv 1$ in $B_r$ \quad and \quad $\left| {{\nabla _H}\phi} \right| \le \frac{c}{r}$ in ${{B_{\frac{{3r}}{2}}}}$.
\end{center}
Taking the test function $\varphi \left( \xi  \right): = \frac{{\phi^q\left( \xi  \right)}}{{{g_{2r}}\left( {u\left( \xi  \right) + d} \right)}}$, we have from the weak formulation that
\begin{align}\label{eq51}
  0 =& \int_{{B_{2r}}} {\int_{{B_{2r}}} {\Bigg[\frac{{{J_p}\left( {u\left( \xi  \right) - u\left( \eta  \right)} \right)}}{{\| {{\eta ^{ - 1}} \circ \xi } \|_{{\mathbb{H}^n}}^{Q + sp}}}} } \left( {\frac{{{\phi^q}\left( \xi  \right)}}{{{g_{2r}}\left( {\overline{u}( \xi ) } \right)}} - \frac{{\phi^q( \xi)}}{{{g_{2r}}\left( {\overline{u}( \eta)} \right)}}} \right)\nonumber \\
   &  + a\left( {\xi ,\eta } \right)\frac{{{J_q}\left( {u\left( \xi  \right) - u\left( \eta  \right)} \right)}}{{\| {{\eta ^{ - 1}} \circ \xi } \|_{{\mathbb{H}^n}}^{Q + tq}}}\left( {\frac{{\phi^q\left( \xi  \right)}}{{{g_{2r}}\left( {\bar u\left( \xi  \right)} \right)}} - \frac{{\phi^q\left( \xi  \right)}}{{{g_{2r}}\left( {\bar u\left( \eta  \right)} \right)}}} \right)\Bigg]\,d\xi d\eta \nonumber \\
  &+2{\int _{{\mathbb{H}^n}\backslash {B_{2r}}}}\int_{{B_{2r}}} {\left[\frac{{{J_p}\left( {u\left( \xi  \right) - u\left( \eta  \right)} \right)}}{{\| {{\eta ^{ - 1}} \circ \xi } \|_{{\mathbb{H}^n}}^{Q + sp}}} + a\left( {\xi ,\eta } \right)\frac{{{J_q}\left( {u\left( \xi  \right) - u\left( \eta  \right)} \right)}}{{\| {{\eta ^{ - 1}} \circ \xi } \|_{{\mathbb{H}^n}}^{Q + tq}}}\right]\frac{\phi^q(\xi)}{g_{2r}(\overline{u}(\xi))}\,d\xi d\eta } \nonumber \\
  =:I_1+I_2
\end{align}
with $\bar u: = u + d$.

In what follows, we deal with $I_1$ in the case $\bar u\left( \xi  \right) \ge \bar u\left( \eta  \right)$ that is divided into two subcases:
\begin{equation}\label{eq52-1}
\bar u\left( \xi  \right) \ge \bar u\left( \eta  \right)\ge \frac{1}{2}\bar u\left( \xi  \right)
\end{equation}
and
\begin{equation}\label{eq52-2}
\bar u\left( \xi  \right) \ge 2\bar u\left( \eta  \right).
\end{equation}
If \eqref{eq52-1} occurs, we first observe that
\begin{align}\label{eq53}
 &\quad  \frac{{{\phi ^q}\left( \xi  \right)}}{{{g_{2r}}\left( {\bar u\left( \xi  \right)} \right)}} - \frac{{{\phi^q}\left( \eta  \right)}}{{{g_{2r}}\left( {\bar u\left( \eta  \right)} \right)}}\nonumber \\
   &\le  \frac{c{{\phi^{q - 1}}\left( \xi  \right)\sup_{B_{3R/2}} \left| {{\nabla _H}\phi} \right|{{\| {{\eta ^{ - 1}} \circ \xi } \|}_{{\mathbb{H}^n}}}}}{{{g_{2r}}\left( {\bar u\left( \eta  \right)} \right)}} +{\phi^q}\left( \xi  \right)\int_0^1 {\frac{d}{{d\sigma }}\left( {g_{2r}^{ - 1}\left( {\sigma \bar u\left( \xi  \right) + (1 - \sigma )\bar u\left( \eta  \right)} \right)} \right)\,d\sigma } \nonumber \\
   & \le \frac{c{{\phi^{q - 1}}\left( \xi  \right){r^{ - 1}}{{\| {{\eta ^{ - 1}} \circ \xi } \|}_{{\mathbb{H}^n}}}}}{{{g_{2r}}\left( {\bar u\left( \eta  \right)} \right)}} - \frac{{\left( {p - 1} \right){\phi^q}\left( \xi  \right)\left( {\bar u\left( \xi  \right) - \bar u\left( \eta  \right)} \right)}}{{{2^q}{G_{2r}}\left( {\bar u\left( \eta  \right)} \right)}},
\end{align}
where we can see that the first inequality holds naturally when $\phi(\xi) \le \phi(\eta)$. Here we have used
\[\int_0^1 {\frac{d}{{d\sigma }}\left( {g_{2r}^{ - 1}\left( {\sigma \bar u\left( \xi  \right) + (1 - \sigma )\bar u\left( \eta  \right)} \right)} \right)\,d\sigma }  \ge \frac{{\left( {p - 1} \right)\left( {\bar u\left( \xi  \right) - \bar u\left( \eta  \right)} \right)}}{{{G_{2r}}\left( {\bar u\left( \xi  \right)} \right)}} \ge \frac{{\left( {p - 1} \right)\left( {\bar u\left( \xi  \right) - \bar u\left( \eta  \right)} \right)}}{{{2^q}{G_{2r}}\left( {\bar u\left( \eta  \right)} \right)}},\]
the details of which can be found in \cite{BOS22}. Then, combining \eqref{eq53} and Young's inequality yields
\begin{align}\label{eq54}
   F\left( {\xi ,\eta } \right)&: = \left( {\frac{{{J_p}\left( {u\left( \xi  \right) - u\left( \eta  \right)} \right)}}{{\| {{\eta ^{ - 1}} \circ \xi } \|_{{\mathbb{H}^n}}^{sp}}} + a\left( {\xi ,\eta } \right)\frac{{{J_q}\left( {u\left( \xi  \right) - u\left( \eta  \right)} \right)}}{{\| {{\eta ^{ - 1}} \circ \xi } \|_{{\mathbb{H}^n}}^{tq}}}} \right)\left( {\frac{{{\phi^q}\left( \xi  \right)}}{{{g_{2r}}\left( {\bar u\left( \xi  \right)} \right)}} - \frac{{{\phi^q}( \eta)}}{{{g_{2r}}\left( {\bar u\left( \eta  \right)} \right)}}} \right)\nonumber \\
   &\le  \frac{{c{\phi^{q - 1}}\left( \xi  \right){r^{ - 1}}{{\| {{\eta ^{ - 1}} \circ \xi } \|}_{{\mathbb{H}^n}}}\bar u\left( \eta  \right)}}{{{G_{2r}}\left( {\bar u\left( \eta  \right)} \right)}}\left( {\frac{{{{\left| {\bar u\left( \xi  \right) - \bar u\left( \eta  \right)} \right|}^{p - 1}}}}{{\| {{\eta ^{ - 1}} \circ \xi } \|_{{\mathbb{H}^n}}^{sp}}} + a\left( {\xi ,\eta } \right)\frac{{{{\left| {\bar u\left( \xi  \right) - \bar u\left( \eta  \right)} \right|}^{q - 1}}}}{{\| {{\eta ^{ - 1}} \circ \xi } \|_{{\mathbb{H}^n}}^{tq}}}} \right)\nonumber \\
  &\quad - \frac{{\left( {p - 1} \right){\phi^q}\left( \xi  \right)H\left( {\xi ,\eta ,\bar u\left( \xi  \right) - \bar u\left( \eta  \right)} \right)}}{{{2^q}{G_{2r}}\left( {\bar u\left( \eta  \right)} \right)}} \nonumber \\
   &\le \frac{{\varepsilon {\phi^{\frac{{\left( {q - 1} \right)p}}{{p - 1}}}}\left( \xi  \right){{\left| {\bar u\left( \xi  \right) - \bar u\left( \eta  \right)} \right|}^p}}}{{{G_{2r}}\left( {\bar u\left( \eta  \right)} \right)\| {{\eta ^{ - 1}} \circ \xi } \|_{{\mathbb{H}^n}}^{sp}}} + a\left( {\xi ,\eta } \right)\frac{{\varepsilon {\phi^q}\left( \xi  \right){{\left| {\bar u\left( \xi  \right) - \bar u\left( \eta  \right)} \right|}^q}}}{{{G_{2r}}\left( {\bar u\left( \eta  \right)} \right)\| {{\eta ^{ - 1}} \circ \xi } \|_{{\mathbb{H}^n}}^{tq}}} \nonumber \\
   &\quad - \frac{{\left( {p - 1} \right){\phi^q}\left( \xi  \right)H\left( {\xi ,\eta ,\bar u\left( \xi  \right) - \bar u\left( \eta  \right)} \right)}}{{{2^q}{G_{2r}}\left( {\bar u\left( \eta  \right)} \right)}}\nonumber \\
   &\quad + c\left( \varepsilon  \right)\frac{{{r^{ - p}}\| {{\eta ^{ - 1}} \circ \xi } \|_{{\mathbb{H}^n}}^p{{\left| {\bar u\left( \eta  \right)} \right|}^p}}}{{{G_{2r}}\left( {\bar u\left( \eta  \right)} \right)\| {{\eta ^{ - 1}} \circ \xi } \|_{{\mathbb{H}^n}}^{sp}}} + c\left( \varepsilon  \right)a_{2r}^ + \frac{{{r^{ - q}}\| {{\eta ^{ - 1}} \circ \xi } \|_{{\mathbb{H}^n}}^q{{\left| {\bar u\left( \eta  \right)} \right|}^q}}}{{{G_{2r}}\left( {\bar u\left( \eta  \right)} \right)\| {{\eta ^{ - 1}} \circ \xi } \|_{{\mathbb{H}^n}}^{tq}}}\nonumber \\
    &\le  - \frac{{\left( {p - 1} \right){\phi^q}\left( \xi  \right)H\left( {\xi ,\eta ,\bar u\left( \xi  \right) - \bar u\left( \eta  \right)} \right)}}{{{2^{q + 1}}{G_{2r}}\left( {\bar u\left( \eta  \right)} \right)}} + c\frac{{{r^{p\left( {s - 1} \right)}}}}{{\| {{\eta ^{ - 1}} \circ \xi } \|_{{\mathbb{H}^n}}^{p\left( {s - 1} \right)}}} + c\frac{{{r^{q\left( {t - 1} \right)}}}}{{\| {{\eta ^{ - 1}} \circ \xi } \|_{{\mathbb{H}^n}}^{q\left( {t - 1} \right)}}},
\end{align}
where $\varepsilon $ was chosen as $\frac{{p - 1}}{{{2^{q + 1}}}}$ and $\frac{{\left( {q - 1} \right)p}}{{p - 1}} > q$ and $c>0$ is independent of $a$. We proceed to evaluate ${{G_{2r}}\left( {\bar u\left( \eta  \right)} \right)}$. For $\xi,\eta \in B_{2r}$, recalling the H\"{o}lder continuity of $a$, we get
\[a_{2r}^ +  = a_{2r}^ +  - a\left( {\xi ,\eta } \right) + a\left( {\xi ,\eta } \right) \le 2{\left[ a \right]_\alpha }{\left( {4r} \right)^\alpha } + a\left( {\xi ,\eta } \right).\]
Thus this implies by the facts that $r\le 1$ and $tq \le sp+\alpha$ that
\begin{align}\label{eq55}
   {G_{2r}}\left( {\bar u\left( \eta  \right)} \right) & \le \frac{{{{\bar u}^p}\left( \eta  \right)}}{{{{\left( {2r} \right)}^{sp}}}} + 2{\left[ a \right]_\alpha }{\left( {4r} \right)^\alpha }\frac{{{{\bar u}^q}\left( \eta  \right)}}{{{{\left( {2r} \right)}^{tq}}}} + a\left( {\xi ,\eta } \right)\frac{{{{\bar u}^q}\left( \eta  \right)}}{{{{\left( {2r} \right)}^{tq}}}}\nonumber \\
   &  \le \left( {1 + 8{{\left[ a \right]}_\alpha }{r^{\alpha  + sp - tq}}\| u \|_{{L^\infty }\left( {\Omega '} \right)}^{q - p}} \right)\frac{{{{\bar u}^p}\left( \eta  \right)}}{{{{\left( {2r} \right)}^{sp}}}} + a\left( {\xi ,\eta } \right)\frac{{{{\bar u}^q}\left( \eta  \right)}}{{{{\left( {2r} \right)}^{tq}}}}\nonumber\\
   & \le c\left( {1 + \| u \|_{{L^\infty }\left( {\Omega '} \right)}^{q - p}} \right){H_{2r}}\left( {\xi ,\eta ,\bar u\left( \eta  \right)} \right).
\end{align}

Next, we will obtain an estimate on $\log \bar u$. It is easy to find
\[\log \frac{{\bar u\left( \xi  \right)}}{{\bar u\left( \eta  \right)}} = \int_0^1 {\frac{{\bar u\left( \xi  \right) - \bar u\left( \eta  \right)}}{{\bar u\left( \eta  \right) + \sigma \left( {\bar u\left( \xi  \right) - \bar u\left( \eta  \right)} \right)}}}\, d\sigma  \le \frac{(\bar u( \xi)-\bar u(\eta))/\|\eta ^{-1}\circ \xi \|_{{\mathbb{H}^n}}^s}
 {\bar u(\eta)/(2r)^s}
 \cdot \frac{\|\eta ^{ - 1}\circ \xi\|_{{\mathbb{H}^n}}^s}{(2r)^s},\]
so there holds that
 \begin{align}\label{eq56}
\log \frac{\bar u(\xi)}{u(\eta)}&\le\frac{\|\eta ^{ - 1} \circ \xi \|_{\mathbb{H}^n}^s}{(2r)^s}\cdot\left[\frac{\left(\frac{\bar u(\xi)-\bar u(\eta)}{\|\eta^{-1}\circ\xi\|_{\mathbb{H}^n}^s}\right)^p+a(\xi,\eta)\left(\frac{\bar u(\xi)-\bar u(\eta)}{\|\eta^{-1}\circ\xi\|_{\mathbb{H}^n}^s}\right)^q\|\eta^{-1}\circ\xi\|_{\mathbb{H}^n}^{-(t-s)q}}{\left(\frac{\bar u(\eta)}{(2r)^s}\right)^p+a(\xi,\eta)\left(\frac{\bar u(\eta)}{(2r)^s}\right)^q\|\eta^{-1}\circ\xi\|_{\mathbb{H}^n}^{-(t-s)q}}+ 1\right] \nonumber\\
    &\le  \frac{{cH\left( {\xi ,\eta ,\bar u\left( \xi  \right) - \bar u\left( \eta  \right)} \right)}}{{{H_{2r}}\left( {\xi ,\eta ,\bar u\left( \eta  \right)} \right)}} + \frac{{\| {{\eta ^{ - 1}} \circ \xi } \|_{{\mathbb{H}^n}}^s}}{{{{\left( {2r} \right)}^s}}},
 \end{align}
where we need to note ${\| {{\eta ^{ - 1}} \circ \xi } \|_{{\mathbb{H}^n}}} \le 4r$. It follows from \eqref{eq54}-\eqref{eq56} that
\[F\left( {\xi ,\eta } \right) \le  - \frac{{{\phi^q}\left( \xi  \right)}}{{cK}}\log \frac{{\bar u\left( \xi  \right)}}{{\bar u\left( \eta  \right)}} + \frac{{c\| {{\eta ^{ - 1}} \circ \xi } \|_{{\mathbb{H}^n}}^s}}{{{{\left( {2r} \right)}^s}}} + \frac{{c\| {{\eta ^{ - 1}} \circ \xi } \|_{{\mathbb{H}^n}}^{p(1 - s)}}}{{{{\left( {2r} \right)}^{p(1 - s)}}}} + \frac{{c\| {{\eta ^{ - 1}} \circ \xi } \|_{{\mathbb{H}^n}}^{q(1 - t)}}}{{{{\left( {2r} \right)}^{q(1 - t)}}}}.\]

Second, we in the case \eqref{eq52-2} tackle the integral $I_1$. Applying Lemma \ref{Le26} and the relation $\bar u\left( \xi  \right) \ge 2\bar u\left( \eta  \right)$, we could derive
\begin{align*}
   \frac{{{\phi^q}\left( \xi  \right)}}{{{g_{2r}}\left( {\bar u\left( \xi  \right)} \right)}} - \frac{{{\phi^q}( \eta)}}{{{g_{2r}}\left( {\bar u\left( \eta  \right)} \right)}}&\le  \frac{{{\phi^q}\left( \xi  \right) - {\phi^q}\left( \eta  \right)}}{{{g_{2r}}\left( {\bar u\left( \xi  \right)} \right)}} + {\phi ^q}\left( \eta  \right)\left( {\frac{1}{{{g_{2r}}\left( {2\bar u\left( \eta  \right)} \right)}} - \frac{1}{{{g_{2r}}\left( {\bar u\left( \eta  \right)} \right)}}} \right)\\
    &\le\frac{{\varepsilon {\phi^q}\left( \eta  \right) + c\left( \varepsilon  \right){{\left| {\phi\left( \xi  \right) - \phi\left( \eta  \right)} \right|}^q}}}{{{g_{2r}}\left( {\bar u\left( \xi  \right)} \right)}} - \frac{{{2^{p - 1}} - 1}}{{{2^{p - 1}}}}\frac{{\phi ^q( \eta)}}{{{g_{2r}}\left( {\bar u\left( \eta  \right)} \right)}}\\
     &\le\frac{{c{{\left| {\phi( \xi) - \phi( \eta)} \right|}^q}}}{{{g_{2r}}\left( {\bar u\left( \xi  \right)} \right)}} - \frac{{\left( {{2^{p - 1}} - 1} \right){\phi^q}\left( \eta  \right)}}{{{2^p}{g_{2r}}\left( {\bar u\left( \eta  \right)} \right)}}
\end{align*}
with $\varepsilon  = \frac{{{2^{p - 1}} - 1}}{{{2^p}}}$. Thereby, it holds that
\begin{align*}
F\left( {\xi ,\eta } \right)& \le \frac{{ch\left( {\xi ,\eta ,\bar u\left( \xi  \right) - \bar u\left( \eta  \right)} \right){{| {\phi( \xi) -\phi( \eta)}|}^q}}}{{{g_{2r}}\left( {\bar u\left( \xi  \right)} \right)}} - \frac{{h\left( {\xi ,\eta ,\bar u\left( \xi  \right) - \bar u\left( \eta  \right)} \right){\phi ^q}\left( \eta  \right)}}{{c{g_{2r}}\left( {\bar u\left( \eta  \right)} \right)}} \\
   & \le \frac{{c{{\left( {2r} \right)}^{ - q}}\| {{\eta ^{ - 1}} \circ \xi } \|_{{\mathbb{H}^n}}^qh\left( {\xi ,\eta ,\bar u\left( \xi  \right) - \bar u\left( \eta  \right)} \right)}}{{{g_{2r}}\left( {\bar u\left( \xi  \right)} \right)}} - \frac{{h\left( {\xi ,\eta ,\bar u\left( \xi  \right) - \bar u\left( \eta  \right)} \right){\phi^q}\left( \eta  \right)}}{{cK{h_{2r}}\left( {\xi ,\eta ,\bar u\left( \eta  \right)} \right)}}.
\end{align*}
Here $F\left( {\xi ,\eta } \right)$ is the same as that in \eqref{eq54} and the estimate for ${{g_{2r}}\left( {\bar u\left( \eta  \right)} \right)}$ is similar to \eqref{eq55}. Moreover, via $\bar u\left( \xi  \right) \ge 2\bar u\left( \eta  \right)\ge 0$ in $B_{2r}$,
\begin{align*}
   \frac{{h\left( {\xi ,\eta ,\bar u( \xi) - \bar u( \eta)} \right)}}{{{g_{2r}}( {\bar u( \xi)})}}
   \le \frac{\frac{|\bar u(\xi)-\bar u(\eta)|^{p-1}}{\|\eta^{-1} \circ \xi\|_{\mathbb{H}^n}^{sp}}+a(\xi,\eta)\frac{|\bar u(\xi)-\bar u(\eta)|^{q-1}}{\|\eta^{-1} \circ \xi\|_{\mathbb{H}^n}^{tq}}}{\frac{|\bar u(\xi)-\bar u(\eta)|^{p-1}}{(2r)^{sp}}+a_{2r}^ +\frac{|\bar u(\xi)-\bar u(\eta)|^{q-1}}{(2r)^{tq}}}
  \le \frac{{{{\left( {2r} \right)}^{sp}}}}{{\| {{\eta ^{ - 1}} \circ \xi } \|_{{\mathbb{H}^n}}^{sp}}} + \frac{{{{\left( {2r} \right)}^{tq}}}}{{\| {{\eta ^{ - 1}} \circ \xi } \|_{{\mathbb{H}^n}}^{tq}}},
\end{align*}
and further
\[F\left( {\xi ,\eta } \right) \le \frac{{c\| {{\eta ^{ - 1}} \circ \xi } \|_{{\mathbb{H}^n}}^{q - sp}}}{{{{\left( {2r} \right)}^{q - sp}}}} + \frac{{\| {{\eta ^{ - 1}} \circ \xi } \|_{{\mathbb{H}^n}}^{q\left( {1 - t} \right)}}}{{{{\left( {2r} \right)}^{q\left( {1 - t} \right)}}}} - \frac{{h\left( {\xi ,\eta ,\bar u\left( \xi  \right) - \bar u\left( \eta  \right)} \right){\phi ^q}\left( \eta  \right)}}{{cK{h_{2r}}\left( {\xi ,\eta ,\bar u\left( \eta  \right)} \right)}}.\]

Now we need to get an estimate on $\log \frac{{\bar u\left( \xi  \right)}}{{\bar u\left( \eta  \right)}}$ under \eqref{eq52-2},
\begin{align*}
   \log \frac{\bar u(\xi)}{\bar u( \eta)}& \le \log \frac{2(\bar u( \xi) - \bar u( \eta))}{\bar u( \eta)}
   \le \frac{c\left( (\bar u(\xi) - \bar u( \eta))/\|\eta ^{ - 1} \circ \xi\|_{\mathbb{H}^n}^s\right)^{p - 1}}{(\bar u(\eta)/(2r)^s)^{p - 1}} \frac{{\| {{\eta ^{ - 1}} \circ \xi } \|_{{\mathbb{H}^n}}^{s\left( {p - 1} \right)}}}{{{{\left( {2r} \right)}^{s\left( {p - 1} \right)}}}}\\
   &\le c\frac{{\| {{\eta ^{ - 1}} \circ \xi } \|_{{\mathbb{H}^n}}^{s\left( {p - 1} \right)}}}{{{{\left( {2r} \right)}^{s\left( {p - 1} \right)}}}}\left[\frac{\left(\frac{\bar u(\xi)-\bar u(\eta)}{\|\eta^{-1}\circ\xi\|_{\mathbb{H}^n}^s}\right)^{p-1}+a(\xi,\eta)\left(\frac{\bar u(\xi)-\bar u(\eta)}{\|\eta^{-1}\circ\xi\|_{\mathbb{H}^n}^s}\right)^{q-1}\|\eta^{-1}\circ\xi\|_{\mathbb{H}^n}^{-(t-s)q}}{\left(\frac{\bar u(\eta)}{(2r)^s}\right)^{p-1}+a(\xi,\eta)\left(\frac{\bar u(\eta)}{(2r)^s}\right)^{q-1}\|\eta^{-1}\circ\xi\|_{\mathbb{H}^n}^{-(t-s)q}}+ 1\right]\\
  &\le \frac{{ch\left( {\xi ,\eta ,\bar u\left( \xi  \right) - \bar u\left( \eta  \right)} \right)}}{{{h_{2r}}\left( {\xi ,\eta ,\bar u\left( \eta  \right)} \right)}} + \frac{{c\| {{\eta ^{ - 1}} \circ \xi } \|_{{\mathbb{H}^n}}^{s\left( {p - 1} \right)}}}{{{{\left( {2r} \right)}^{s\left( {p - 1} \right)}}}},
\end{align*}
where the fact ${\| {{\eta ^{ - 1}} \circ \xi } \|_{{\mathbb{H}^n}}} \le 4r$ was utilized. Note $q\ge p$ and ${\| {{\eta ^{ - 1}} \circ \xi } \|_{{\mathbb{H}^n}}} \le 4r$ again. It is clear to obtain
\[F\left( {\xi ,\eta } \right) \le  - \frac{{{\phi^q}\left( \xi  \right)}}{{cK}}\log \frac{{\bar u\left( \xi  \right)}}{{\bar u\left( \eta  \right)}} + \frac{{c\| {{\eta ^{ - 1}} \circ \xi } \|_{{\mathbb{H}^n}}^{p\left( {1 - s} \right)}}}{{{{\left( {2r} \right)}^{p\left( {1 - s} \right)}}}} + \frac{{c\| {{\eta ^{ - 1}} \circ \xi } \|_{{\mathbb{H}^n}}^{q(1 - t)}}}{{{{\left( {2r} \right)}^{q(1 - t)}}}} + \frac{{c\| {{\eta ^{ - 1}} \circ \xi } \|_{{\mathbb{H}^n}}^{s(p - 1)}}}{{{{\left( {2r} \right)}^{s(p - 1)}}}}.\]

At this moment, for $\bar u\left( \xi  \right) \ge \bar u\left( \eta  \right)$, the integral $I_1$ is evaluated as
\begin{align}\label{eq57}
I_1 \le &   - \frac{1}{{cK}}\int_{{B_{2r}}} {\int_{{B_{2r}}} {\min \left\{ {{\phi^q}( \xi),{\phi^q}(\eta)} \right\}\left| {\log \frac{{\bar u(\xi)}}{{\bar u(\eta)}}} \right|\frac{{d\xi d\eta }}{{\| {{\eta ^{-1}}\circ\xi}\|_{{\mathbb{H}^n}}^Q}}} } \nonumber \\
   &  + c\int_{{B_{2r}}} {\int_{{B_{2r}}} {\left( {\frac{{\| {{\eta ^{ - 1}} \circ \xi }\|_{{\mathbb{H}^n}}^{p(1 - s)}}}{{{r^{p(1 - s)}}}} + \frac{{\| {{\eta ^{ - 1}} \circ \xi }\|_{{\mathbb{H}^n}}^{q(1 - t)}}}{{{r^{q(1 - t)}}}} + \frac{{\| {{\eta ^{ - 1}} \circ \xi } \|_{{\mathbb{H}^n}}^{s(p - 1)}}}{{{r^{s( {p - 1})}}}} + \frac{{\| {{\eta ^{ - 1}} \circ \xi } \|_{{\mathbb{H}^n}}^s}}{{{r^s}}}}\right)\frac{{d\xi d\eta }}{{\| {{\eta ^{ - 1}} \circ \xi } \|_{{\mathbb{H}^n}}^Q}}} }\nonumber \\
    \le  &- \frac{1}{{cK}}\int_{{B_{2r}}} {\int_{{B_{2r}}} {\left| {\log \frac{{\bar u( \xi)}}{{\bar u( \eta)}}} \right|\frac{{d\xi d\eta }}{{\| {{\eta ^{ - 1}} \circ \xi }\|_{{\mathbb{H}^n}}^Q}}} }  + c{r^Q},
\end{align}
where
\begin{align*}
\int_{{B_{2r}}} {\int_{{B_{2r}}} {\frac{{\| {{\eta ^{ - 1}} \circ \xi } \|_{{\mathbb{H}^n}}^l}}{{{r^l}}}\frac{{d\xi d\eta }}{{\| {{\eta ^{ - 1}} \circ \xi } \|_{{\mathbb{H}^n}}^Q}}} } & \le \int_{{B_{2r}}} {\int_{{B_{4r}(\eta)}} {\frac{{\| {{\eta ^{ - 1}} \circ \xi } \|_{{\mathbb{H}^n}}^l}}{{{r^l}}}\frac{{d\xi d\eta }}{{\| {{\eta ^{ - 1}} \circ \xi } \|_{{\mathbb{H}^n}}^Q}}} } \\
&\le \frac{c}{r^l}\int_{{B_{2r}}}\int^{4r}_0\rho^{l-1}\,d\rho d\eta \le c{r^Q}.
\end{align*}
Furthermore, if $\bar u\left( \xi  \right) < \bar u\left( \eta  \right)$, the same estimate still holds true through exchanging the roles of $\xi$ and $\eta$.

For the second contribution $I_2$ in \eqref{eq51}, we first observe that if $\eta  \in {B_R}$, then ${( {u\left( \xi  \right) - u\left( \eta  \right)} )_ + } \le u\left( \xi  \right) + d$ by $u\left( \eta  \right) \ge 0$, and that if $\eta  \in {\mathbb{H}^n}\backslash {B_R}$, then ${\left( {u\left( \xi  \right) - u\left( \eta  \right)} \right)_ + } \le u\left( \xi  \right) + u_-\left( \eta  \right) \le \bar u\left( \xi  \right) + u_-\left( \eta  \right)$. From this and ${\rm supp}\, \phi\subset {B_{\frac{{3r}}{2}}}$, we can evaluate $I_2$ as
\begin{align}\label{eq58}
  I_2 &\le 2{\int _{{B_R}\backslash {B_{2r}}}}\int_{{B_{\frac{{3r}}{2}}}} {\left[\frac{{\left( {u\left( \xi  \right) - u\left( \eta  \right)} \right)_ + ^{p - 1}}}{{\| {{\eta ^{ - 1}} \circ \xi } \|_{{\mathbb{H}^n}}^{Q + sp}}} + a\left( {\xi ,\eta } \right)\frac{{\left( {u\left( \xi  \right) - u\left( \eta  \right)} \right)_ + ^{q - 1}}}{{\| {{\eta ^{ - 1}} \circ \xi } \|_{{\mathbb{H}^n}}^{Q + tq}}}\right]\frac{{d\xi d\eta }}{{{g_{2r}}\left( {\bar u\left( \xi  \right)} \right)}}} \nonumber \\
   &\quad  + 2{\int _{{\mathbb{H}^n}\backslash {B_R}}}\int_{{B_{\frac{{3r}}{2}}}} {\left[\frac{{\left( {u\left( \xi  \right) - u\left( \eta  \right)} \right)_ + ^{p - 1}}}{{\| {{\eta ^{ - 1}} \circ \xi } \|_{{\mathbb{H}^n}}^{Q + sp}}} + a\left( {\xi ,\eta } \right)\frac{{\left( {u\left( \xi  \right) - u\left( \eta  \right)} \right)_ + ^{q - 1}}}{{\| {{\eta ^{ - 1}} \circ \xi } \|_{{\mathbb{H}^n}}^{Q + tq}}}\right]\frac{{d\xi d\eta }}{{{g_{2r}}\left( {\bar u\left( \xi  \right)} \right)}}}\nonumber \\
  & \le c{\int _{{\mathbb{H}^n}\backslash {B_{2r}}}}\int_{{B_{\frac{{3r}}{2}}}} {\frac{{h\left( {\xi ,\eta ,\bar u\left( \xi  \right)} \right)}}{{{g_{2r}}\left( {\bar u\left( \xi  \right)} \right)\| {{\eta ^{ - 1}} \circ \xi }\|_{{\mathbb{H}^n}}^Q}}\,d\xi d\eta }  + c{\int _{{\mathbb{H}^n}\backslash {B_R}}}\int_{{B_{\frac{{3r}}{2}}}} {\frac{{h\left( {\xi ,\eta ,u\left( \eta  \right)} \right)}}{{{g_{2r}}\left( {\bar u\left( \xi  \right)} \right)\| {{\eta ^{ - 1}} \circ \xi } \|_{{\mathbb{H}^n}}^Q}}\,d\xi d\eta } \nonumber \\
   &= :{I_{21}} + {I_{22}}.
\end{align}

We now intend to control preciesly the term $\frac{{h\left( {\xi ,\eta ,\bar u\left( \xi  \right)} \right)}}{{{g_{2r}}\left( {\bar u\left( \xi  \right)} \right)}}$ by some constants. In view of the condition \eqref{eqy110}, there holds that, for $\xi \in B_{2r}$ and $\eta \in \mathbb{H}^n$,
\begin{align}\label{eq59}
   a\left( {\xi ,\eta } \right)& \le a\left( {\xi ,\eta } \right) - a\left( {\xi ,\xi } \right) + a_{2r}^ +   \nonumber \\
   &  \le {\left( {2{{\| a \|}_{{L^\infty }}}} \right)^{1 - \frac{{tq - sp}}{\alpha }}}{\left| {a\left( {\xi ,\eta } \right) - a\left( {\xi ,\xi } \right)} \right|^{\frac{{tq - sp}}{\alpha }}} + a_{2r}^ + \nonumber \\
   & \le c\| {{\eta ^{ - 1}} \circ \xi } \|_{{\mathbb{H}^n}}^{tq - sp} + a_{2r}^ + .
\end{align}
This indicates
\begin{align*}
 {I_{21}}  &  \le c{\int _{{\mathbb{H}^n}\backslash {B_{2r}}}}\int_{{B_{\frac{{3r}}{2}}}} {\frac{{\frac{{{{\bar u}^{p - 1}}\left( \xi  \right) + {{\bar u}^{q - 1}}\left( \xi  \right)}}{{\| {{\eta ^{ - 1}} \circ \xi } \|_{{\mathbb{H}^n}}^{sp}}} + a_{2r}^ + \frac{{{{\bar u}^{q - 1}}\left( \xi  \right)}}{{\| {{\eta ^{ - 1}} \circ \xi } \|_{{\mathbb{H}^n}}^{tq}}}}}{{\frac{{{{\bar u}^{p - 1}}\left( \xi  \right)}}{{\| {{\eta ^{ - 1}} \circ \xi } \|_{{\mathbb{H}^n}}^{sp}}}\frac{{\| {{\eta ^{ - 1}} \circ \xi } \|_{{\mathbb{H}^n}}^{sp}}}{{{{\left( {2r} \right)}^{sp}}}} + a_{2r}^ + \frac{{{{\bar u}^{q - 1}}\left( \xi  \right)}}{{\| {{\eta ^{ - 1}} \circ \xi } \|_{{\mathbb{H}^n}}^{tq}}}\frac{{\| {{\eta ^{ - 1}} \circ \xi } \|_{{\mathbb{H}^n}}^{tq}}}{{{{\left( {2r} \right)}^{tq}}}}}}\frac{{d\xi d\eta }}{{\| {{\eta ^{ - 1}} \circ \xi } \|_{\mathbb{H}{^n}}^Q}}}  \\
  &  \le cK{\int _{{\mathbb{H}^n}\backslash {B_{2r}}}}\int_{{B_{\frac{{3r}}{2}}}} {\frac{{{{\left( {\frac{r}{2}} \right)}^{sp}}}}{{\| {{\eta ^{ - 1}} \circ \xi } \|_{{\mathbb{H}^n}}^{Q + sp}}}\,d\xi d\eta }
\end{align*}
by virtue of ${\| {{\eta ^{ - 1}} \circ \xi } \|_{{\mathbb{H}^n}}} > \frac{r}{2}$. For $\xi  \in {B_{\frac{{3r}}{2}}}$ and $\eta  \in \mathbb{H}^n\backslash {B_{2r}}$, we derive that
\begin{align}\label{eq510}
   {\| {{\eta ^{ - 1}} \circ {\xi _0}} \|_{{\mathbb{H}^n}}} & \le {\| {{\eta ^{ - 1}} \circ \xi } \|_{{\mathbb{H}^n}}} + {\| {{\xi ^{ - 1}} \circ {\xi _0}} \|_{{\mathbb{H}^n}}} \nonumber \\
   &= \left( {1 + \frac{{{{\| {{\xi ^{ - 1}} \circ {\xi _0}} \|}_{{\mathbb{H}^n}}}}}{{{{\| {{\eta ^{ - 1}} \circ \xi } \|}_{{\mathbb{H}^n}}}}}} \right){\| {{\eta ^{ - 1}} \circ \xi } \|_{{\mathbb{H}^n}}} \nonumber\\
   &  \le \left( 1 + \frac{3r/2}{r/2} \right){\| {{\eta ^{ - 1}} \circ \xi } \|_{{\mathbb{H}^n}}} = 4{\| {{\eta ^{ - 1}} \circ \xi }\|_{{\mathbb{H}^n}}},
\end{align}
where for the first inequality we can refer to \cite{C81}. Thus by \cite[Lemma 2.6]{MPP23},
\begin{equation}\label{eq511}
  {I_{21}} \le cK\left| {{B_{\frac{{3r}}{2}}}} \right|{\int _{{\mathbb{H}^n}\backslash {B_{2r}}}}\frac{{{r^{sp}}}}{{\| {{\eta ^{ - 1}} \circ {\xi _0}} \|_{{\mathbb{H}^n}}^{Q + sp}}}\,d\eta  \le cK{r^Q}.
\end{equation}
Let us proceed to examine $I_{22}$. With the aid of \eqref{eq59}, \eqref{eq510} and $u(\xi)\ge 0$ in ${{B_{\frac{{3r}}{2}}}}$,
\begin{align}\label{eq512}
  I_{22} &  \le c{\int _{{\mathbb{H}^n}\backslash {B_R}}}\int_{{B_{\frac{{3r}}{2}}}} {\left( {\frac{{u_ - ^{p - 1}( \eta) + u_ - ^{q - 1}( \eta)}}{{\| {{\eta ^{ - 1}} \circ \xi }\|_{{\mathbb{H}^n}}^{Q + sp}}} + a_{2r}^ + \frac{{u_ - ^{q - 1}( \eta  )}}{{\| {{\eta ^{ - 1}} \circ \xi }\|_{{\mathbb{H}^n}}^{Q + tq}}}} \right){g^{ - 1}}( d)\,d\xi d\eta } \nonumber \\
   &  \le c{r^Q}{g^{ - 1}}( d){\int _{{\mathbb{H}^n}\backslash {B_R}}}\left( {\frac{{u_ - ^{p - 1}( \eta) + u_ - ^{q - 1}( \eta)}}{{\| {{\eta ^{ - 1}} \circ {\xi _0}}\|_{{\mathbb{H}^n}}^{Q + sp}}} + a_{2r}^ + \frac{{u_ - ^{q - 1}(\eta)}}{{\| {{\eta ^{ - 1}} \circ {\xi _0}}\|_{{\mathbb{H}^n}}^{Q + tq}}}} \right)\,d\eta\nonumber \\
   & \le c{r^{Q+sp}}{d^{1 - p}}{\int _{{\mathbb{H}^n}\backslash {B_R}}}\frac{{u_ - ^{p - 1}( \eta) + u_ - ^{q - 1}( \eta)}}{{\| {{\eta ^{ - 1}} \circ {\xi _0}}\|_{{\mathbb{H}^n}}^{Q + sp}}}d\eta  + c{r^{Q + tq}}{d^{1 - q}}{\int _{{\mathbb{H}^n}\backslash {B_R}}}\frac{{u_ - ^{q - 1}( \eta)}}{{\| {{\eta ^{ - 1}} \circ {\xi _0}}\|_{{\mathbb{H}^n}}^{Q + tq}}}\,d\eta,
\end{align}
where we notice $\eta  \in {\mathbb{H}^n}\backslash {B_R} \subset {\mathbb{H}^n}\backslash {B_{2r}}$.

Merging \eqref{eq57}, \eqref{eq58}, \eqref{eq511}, \eqref{eq512} with \eqref{eq51} arrives eventually at the desired estimate with the positive constant $c$ depending upon $n,p,q,s,t,\alpha,[a]_{\alpha} $ and $\| a \|_{L^\infty }$.
\end{proof}

\begin{corollary}\label{Co52}
  Let the assumptions of Lemma \ref{Le51} be in force. Define
\[w: = \min \left\{ {{{\left( {\log \left( {\tau  + d} \right) - \log \left( {u + d} \right)} \right)}_ + },\log b} \right\}\]
with $\tau,d>0$ and $b>1$. Then for the weak solution $u$ of \eqref{eqy11} it holds that
\begin{align*}
   \fint_{{B_r}} {| {w - {{( w)}_r}}|\,d\eta }  \le c{K^2}\left(1 + \frac{r^{sp}}{d^{p -1}}{\int _{{\mathbb{H}^n}\backslash {B_R}}}\frac{{u_ - ^{p - 1}( \eta) + u_ - ^{q - 1}( \eta)}}{{\| {{\eta ^{ - 1}} \circ {\xi _0}} \|_{{\mathbb{H}^n}}^{Q + sp}}}\,d\eta
    + \frac{r^{tq}}{d^{q - 1}}{\int _{{\mathbb{H}^n}\backslash {B_R}}}\frac{{u_ - ^{q - 1}( \eta)}}{{\| {{\eta ^{ - 1}} \circ {\xi _0}}\|_{{\mathbb{H}^n}}^{Q + tq}}}\,d\eta \right),
\end{align*}
where $c>1$ depends on $n,p,q,s,t$ and $\alpha, {\left[ a \right]_\alpha },\| a \|_{L^\infty }$, and $K$ is defined as in Lemma \ref{Le51}.
\end{corollary}

\begin{proof} Notice that, since $w$ is a truncation of $\log (u+d)$,
\begin{align*}
   \fint_{{B_r}} {| {w - {( w)}_r}|\,d\eta } & \le \fint_{{B_r}} {\left| {\fint_{{B_r}} {( {w( \eta) - w( \xi)})\,d\xi } } \right|\,d\eta }  \\
   &  \le \fint_{{B_r}} {\fint_{{B_r}} {| {w( \xi ) - w( \eta)}|\,d\xi } d\eta } \\
   & \le \fint_{{B_r}} {\fint_{{B_r}} {\frac{{\left| {\log \left( {u\left( \xi  \right) + d} \right) - \log \left( {u\left( \eta  \right) + d} \right)} \right|}}{{{\| {{\eta ^{ - 1}} \circ \xi } \|_{{\mathbb{H}^n}}^Q}/{{{\left( {2r} \right)}^Q}}}}\,d\xi } d\eta } \\
   & \le \fint_{{B_r}} {\int_{{B_r}} {\left| {\log \frac{{u\left( \xi  \right) + d}}{{u\left( \eta  \right) + d}}} \right|\frac{{d\xi d\eta }}{{\| {{\eta ^{ - 1}} \circ \xi }\|_{{\mathbb{H}^n}}^Q}}} } .
\end{align*}
Then the desired result is a plain consequence of Lemma \ref{Le51}.
\end{proof}

In the end, we will focus on establishing H\"{o}lder regularity of weak solutions. For this aim, it is sufficient to show an oscillation improvement result, Theorem \ref{thm 5.1}. Before proceeding, let us introduce some notations. For $j\in\mathbb{N}\cup\{0\}$, set
$$
r_j:=\sigma^jr, \quad \sigma\in(0,1/4] ,\quad  B_j:=B_{r_j}(\xi_0) \quad\text{and} \quad 2B_j:=B_{2r_j},
$$
where we fix any ball $B_{2r}(\xi_0)\subset\Omega'\subset\subset\Omega$. Furthermore, define
$$
\omega(r_0):=2\sup_{B_r}|u|+\left(r^{sp}\int_{\mathbb{H}^n\setminus B_r}\frac{|u|^{p-1}+|u|^{q-1}}{\|\xi_0^{-1}\circ\xi\|^{Q+sp}}\,d\xi\right)^\frac{1}{p-1}+\left(r^{tq}\int_{\mathbb{H}^n\setminus B_r}\frac{|u|^{q-1}}{\|\xi_0^{-1}\circ\xi\|^{Q+tq}}\,d\xi\right)^\frac{1}{q-1}
$$
and
$$
\omega(r_j):=\left(\frac{r_j}{r_0}\right)\omega(r_0)=\sigma^{j\beta}\omega(r) \quad\text{for some } 0<\beta<\frac{sp}{q-1}.
$$
Let us point out that $\sigma$ and $\beta$ are to be determined later.

Now we are in a position to prove the following iteration lemma, which suggests $u\in C^{0,\beta}(B_r)$.

\begin{theorem}\label{thm 5.1}
Let $u\in \mathcal{A}(\Omega)\cap L^{q-1}_{sp}(\mathbb{H}^n)$ be a weak solution to \eqref{eqy11}. Under the conditions \eqref{eqy18}, \eqref{eqy19} and \eqref{eqy110} with $tq\le sp+\alpha$, there holds that
$$
\mathop{\rm osc}\limits_{B_j}u\le \omega(r_j) \quad\text{for any } j\in\mathbb{N}\cup\{0\},
$$
where these notations are fixed as above.
\end{theorem}

\begin{proof}
 Argue by induction. The conclusion is obvious for $j=0$ and then assume it holds true for $i\le j$. Now we show this claim for $j+1$. Let us notice the simple fact that either
\begin{equation}
\label{3-14}
\left|2B_{j+1}\cap\left\{u\ge \inf_{B_j}u+\omega(r_j)/2\right\}\right|\ge\frac{1}{2}|2B_{j+1}|
\end{equation}
or
\begin{equation}
\label{3-15}
\left|2B_{j+1}\cap\left\{u<\inf_{B_j}u+\omega(r_j)/2\right\}\right|\ge\frac{1}{2}|2B_{j+1}|.
\end{equation}
Define
\begin{equation*}
u_j=\begin{cases}u-\inf_{B_j}u, &\text{\textmd{if \eqref{3-14} occurs}},\\[2mm]
\sup_{B_j}u-u, &\text{\textmd{if \eqref{3-15} occurs}}.
\end{cases}
\end{equation*}
Obviously, $u_j\ge0$ in $B_j$ and
\begin{equation}
\label{3-16}
|2B_{j+1}\cap\{u_j\geq\omega(r_j)/2\}|\ge\frac{1}{2}|2B_{j+1}|.
\end{equation}
Moreover, $u_j$ is a weak solution to \eqref{eqy11} such that
\begin{equation}
\label{3-17}
\sup_{B_i}|u_j|\le \omega(r_i) \quad\text{for any } i\in\{0,1,2,\cdots,j\}.
\end{equation}

Now we set an auxiliary function
$$
w:=\min\left\{\left[\log\left(\frac{\omega(r_j)/2+d}{u_j+d}\right)\right]_+,k\right\} \quad\text{with } k>0.
$$
Applying Corollary \ref{Co52} derives
\begin{align}
\label{3-18}
&\quad\fint_{2B_{j+1}}|w-(w)_{2B_{j+1}}|\,d\xi \nonumber\\
&\leq CK^2\left(1+d^{1-p}r^{sp}_{j+1}\int_{\mathbb{H}^n\setminus B_j}\frac{|u_j|^{p-1}+|u_j|^{q-1}}{\|\xi_0^{-1}\circ\xi\|_{\mathbb{H}^n}^{Q+sp}}\,d\xi+d^{1-q}r^{tq}_{j+1}\int_{\mathbb{H}^n\setminus B_j}\frac{|u_j|^{q-1}}{\|\xi_0^{-1}\circ\xi\|_{\mathbb{H}^n}^{Q+tq}}\,d\xi\right)
\end{align}
with $K$ defined as in Lemma \ref{Le51}. We evaluate the second integral at the right-hand side. By means of \eqref{3-17} and the definition of $\omega(r_0)$,
\begin{align}\label{3-19}
r^{tq}_{j+1}\int_{\mathbb{H}^n\setminus B_j}\frac{|u_j|^{q-1}}{\|\xi_0^{-1}\circ\xi\|_{\mathbb{H}^n}^{Q+tq}}\,d\xi &=r^{tq}_j\sum^j_{i=1}\int_{B_{i-1}\setminus B_i}\frac{|u_j|^{q-1}}{\|\xi_0^{-1}\circ\xi\|_{\mathbb{H}^n}^{Q+tq}}\,d\xi
+r^{tq}_j\int_{\mathbb{H}^n\setminus B_0}\frac{|u_j|^{q-1}}{\|\xi_0^{-1}\circ\xi\|_{\mathbb{H}^n}^{Q+tq}}\,d\xi \nonumber\\
&\leq\sum^j_{i=1}\omega(r_{i-1})^{q-1}\left(\frac{r_j}{r_i}\right)^{tq}+Cr^{tq}_j\int_{\mathbb{H}^n\setminus B_0}\frac{|u|^{q-1}+(\sup_{B_0}|u|)^{q-1}}{\|\xi_0^{-1}\circ\xi\|_{\mathbb{H}^n}^{Q+tq}}\,d\xi \nonumber\\
&\leq C\sum^j_{i=1}\left(\frac{r_j}{r_i}\right)^{tq}\omega(r_{i-1})^{q-1} \nonumber\\
&\le C\frac{4^{tq-\beta(q-1)}}{(tq-\beta(q-1))\log4}\sigma^{-\beta(q-1)}\omega(r_j)^{q-1},
\end{align}
where we used the fact that $\beta<\frac{sp}{q-1}\left(\le \frac{tq}{q-1}\right)$. Analogously,
\begin{align}\label{3-20}
r^{sp}_j\int_{\mathbb{H}^n\setminus B_j}\frac{|u_j|^{p-1}+|u_j|^{q-1}}{\|\xi_0^{-1}\circ\xi\|_{\mathbb{H}^n}^{Q+sp}}\,d\xi &\le C(1+\|u\|^{q-p}_{L^\infty(\Omega')})\sum^j_{i=1}\left(\frac{r_j}{r_i}\right)^{sp}\omega(r_{i-1})^{p-1} \nonumber\\
&\le CN\sigma^{-\beta(p-1)}\omega(r_j)^{p-1}
\end{align}
with $\beta<\frac{sp}{q-1}\left(\le \frac{sp}{p-1}\right)$, where $N:=1+\|u\|^{q-p}_{L^\infty(\Omega')}$ and the derivation of $\|u\|^{q-p}_{L^\infty(\Omega')}$ is from the term $|u_j|^{q-1}$, and $C>0$ depends on $n,p,s$ and the difference of $\frac{sp}{p-1}$ and $\beta$. Combining \eqref{3-19}, \eqref{3-20} with \eqref{3-18} and remembering $\frac{r_{j+1}}{r_j}=\sigma$, we get
\begin{align*}
\fint_{2B_{j+1}}|w-(w)_{2B_{j+1}}|\,d\xi
\leq CK^2\left(1+Nd^{1-p}\sigma^{sp-\beta(p-1)}\omega(r_j)^{p-1}+d^{1-q}\sigma^{tq-\beta(q-1)}\omega(r_j)^{q-1}\right),
\end{align*}
where $C$ depends on $n,p,q,s,t,\beta$. 

In what follows, picking
$$
d:=\sigma^{\frac{sp}{q-1}-\beta}\omega(r_j)
$$
and recalling $\omega(r_j)=\sigma^{j\beta}\omega(r_0)$, we find
\begin{align*}
\fint_{2B_{j+1}}|w-(w)_{2B_{j+1}}|\,d\xi &\leq CK^2\left[1+N\sigma^{\left({\frac{sp}{q-1}-\beta}\right)(1-p)+\left({\frac{sp}{p-1}-\beta}\right)(p-1)}
+\sigma^{\left({\frac{sp}{q-1}-\beta}\right)(1-q)+\left({\frac{tq}{q-1}-\beta}\right)(q-1)}\right] \nonumber\\
&\leq CN^3,
\end{align*}
where $C$ depends on $n,p,q,s,t,\alpha,[a]_\alpha,\|a\|_{L^\infty}$ and the difference of $\beta$ and $\frac{tq}{q-1}$, and $\frac{sp}{p-1}$. Here we need utilize the definition of $K$ as in Lemma \ref{Le51}, and $\omega(r_j)\le 2\|u\|_{L^\infty(\Omega')}$. From the last inequality,
\begin{equation*}
\frac{|2B_{j+1}\cap\{w=k\}|}{|2B_{j+1}|}\le \frac{CN^3}{k}.
\end{equation*}
We refer to \cite[page 1296]{DKP16} for the details. By taking
$$
k=\log\left(\frac{\omega(r_j)/2+\varepsilon\omega(r_j)}{3\varepsilon\omega(r_j)}\right)=\log\left(\frac{1/2+\varepsilon}{3\varepsilon}\right)\approx\log\frac{1}{\varepsilon}
$$
with $\varepsilon:=\sigma^{\frac{sp}{q-1}-\beta}$, it holds that
\begin{equation}
\label{3-22}
\frac{|2B_{j+1}\cap\{u_j\leq2\varepsilon\omega(r_j)\}|}{|2B_{j+1}|}\le \frac{CN^3}{k}\le \frac{C_{\rm log}N^3}{\log\frac{1}{\sigma}}
\end{equation}
for the constant $C_{\rm log}>0$ depending on $n,p,q,s,t,\alpha,[a]_\alpha,\|a\|_{L^\infty}$ and $\beta$.

At this moment, we are going to perform a suitable iteration. For each $i=0,1,\cdots$, let
$$
\rho_i=r_{j+1}+2^{-i}r_{j+1}, \quad \hat{\rho_i}=\frac{\rho_i+3\rho_{i+1}}{4},  \quad \tilde{\rho_i}=\frac{3\rho_i+\rho_{i+1}}{4}
$$
and the corresponding balls
$$
B^i=B_{\rho_i}, \quad \hat{B}^i=B_{\hat{\rho_i}}, \quad \tilde{B}^i=B_{\tilde{\rho_i}}.
$$
Then take the cut-off functions $\psi_i\in C^\infty_0(\tilde{B}^i)$ such that
$$
0\le \psi_i\le 1, \quad \psi_i\equiv 1 \text{ in } \hat{B}^{i} \quad\text{and} \quad |\nabla_H\psi_i|\le 2^{i+2}r^{-1}_{j+1}.
$$
Besides, set
$$
k_i=(1+2^{-i})\varepsilon\omega(r_j), \quad w_i=(k_i-u_j)_+
$$
and
$$
A_i=\frac{|B^i\cap \{u_j\leq k_i\}|}{|B^i|}=\frac{|B^i\cap \{w_j\ge0\}|}{|B^i|}.
$$
Observe the apparent facts that
$$
r_{j+1}\leq \rho_{i+1}<\hat{\rho}_i<\tilde{\rho}_i<\rho_i\le2r_{j+1},\quad 0\le w_i\le k_i\le2\varepsilon\omega(r_j),
$$
and denote
$$
a^+_{j+1}:=\sup_{B_{2r_{j+1}}\times B_{2r_{j+1}}}a(\cdot,\cdot),\ \ \ a^-_{j+1}:=\inf_{B_{2r_{j+1}}\times B_{2r_{j+1}}}a(\cdot,\cdot),\ \ \
\overline{G}(\tau):=\frac{\tau^p}{r^{sp}_{j+1}}+a^+_{j+1}\frac{\tau^q}{r^{tq}_{j+1}}.
$$

With the help of Caccioppoli inequality (Lemma \ref{Le42}), we derive
\begin{align}
\label{3-23}
\fint_{\hat{B}^i}\int_{\hat{B}^i}\frac{H(\xi,\eta,|w_i(\xi)-w_i(\eta)|)}{\|\eta^{-1}\circ\xi\|^Q_{\mathbb{H}^n}}\,d\xi d\eta &\le C\fint_{B^i}\int_{B^i}\frac{H(\xi,\eta,(w_i(\xi)+w_i(\eta))|\psi_i(\xi)-\psi_i(\eta)|)}{\|\eta^{-1}\circ\xi\|^Q_{\mathbb{H}^n}}\,d\xi d\eta  \nonumber\\
&\quad +C\fint_{B^i}w_i\psi_i^q\,d\xi\left(\sup_{\eta\in \tilde{B}^i}\int_{\mathbb{H}^n\setminus B^i}\frac{h(\xi,\eta,w_i(\xi))}{\|\eta^{-1}\circ\xi\|^Q_{\mathbb{H}^n}}\,d\xi\right) \nonumber\\
&=:J_1+J_2.
\end{align}
Via the definition of $w_i$ and $\psi_i$, $J_1$ is evaluated as
\begin{align}
\label{3-25}
J_1&\leq C\frac{2^{ip}k_i^p}{r_{j+1}^p}\int_{B^i\cap \{u_j\le k_i\}}\fint_{B^i}\|\eta^{-1}\circ\xi\|_{\mathbb{H}^n}^{-Q+(1-s)p}\,d\xi d\eta \nonumber\\
&\quad+Ca^+_{j+1}\frac{2^{iq}k_i^q}{r_{j+1}^q}\int_{B^i\cap \{u_j\le k_i\}}\fint_{B^i}\|\eta^{-1}\circ\xi\|_{\mathbb{H}^n}^{-Q+(1-t)q}\,d\xi d\eta \nonumber\\
&\leq C2^{iq}\overline{G}(k_i)A_i
\end{align}
and moreover, we have
$$
\fint_{B^i}w_i\psi_i^q\,d\xi\le Ck_iA_i.
$$
As for the nonlocal integral in $J_2$, we first note that if $\eta\in\tilde{B}^i$ and $\xi\in\mathbb{H}^n\setminus B^i$, then
$$
\|\xi_0^{-1}\circ\xi\|_{\mathbb{H}^n}\le\left(1+\frac{\|\xi_0^{-1}\circ\eta\|_{\mathbb{H}^n}}{\|\eta^{-1}\circ\xi\|_{\mathbb{H}^n}}\right)\|\eta^{-1}\circ\xi\|_{\mathbb{H}^n}
\le2^{i+4}\|\eta^{-1}\circ\xi\|_{\mathbb{H}^n}.
$$
Furthermore, $w_i\le k_i\le2\varepsilon\omega(r_j)$ in $B_j$ (by $u_j\ge0$ in $B_j$), and $w_i\le k_i+|u|$ in $\mathbb{H}^n\setminus B_j$. In a similar way to treat $I_2$ in the proof of Lemma \ref{Le51}, we applying \eqref{3-19}, \eqref{3-20}, the definition of $\varepsilon$ and $B_{j+1}\subset B^i$ to derive
\begin{align*}
&\quad \sup_{\eta\in \tilde{B}^i}\int_{\mathbb{H}^n\setminus B^i}\frac{h(\xi,\eta,w_i(\xi))}{\|\eta^{-1}\circ\xi\|^Q_{\mathbb{H}^n}}\,d\xi \\
&\leq \sup_{\eta\in \tilde{B}^i}\int_{\mathbb{H}^n\setminus B^i}\frac{w_i^{p-1}+w_i^{q-1}}{\|\eta^{-1}\circ\xi\|^{Q+sp}_{\mathbb{H}^n}}
+a^+_{j+1}\frac{w_i^{q-1}}{\|\eta^{-1}\circ\xi\|^{Q+tq}_{\mathbb{H}^n}}\,d\xi\\
&\leq C2^{i(Q+sp+tq)}\int_{\mathbb{H}^n\setminus B_{j+1}}\frac{w_i^{p-1}+w_i^{q-1}}{\|\xi_0^{-1}\circ\xi\|^{Q+sp}_{\mathbb{H}^n}}
+a^+_{j+1}\frac{w_i^{q-1}}{\|\xi_0^{-1}\circ\xi\|^{Q+tq}_{\mathbb{H}^n}}\,d\xi\\
&\leq C2^{i(Q+sp+tq)}\int_{\mathbb{H}^n\setminus B_{j}}\frac{|u_j|^{p-1}+|u_j|^{q-1}}{\|\xi_0^{-1}\circ\xi\|^{Q+sp}_{\mathbb{H}^n}}+
a^+_{j+1}\frac{|u_j|^{q-1}}{\|\xi_0^{-1}\circ\xi\|^{Q+tq}_{\mathbb{H}^n}}\,d\xi\\
&\quad+ C2^{i(Q+sp+tq)}\int_{\mathbb{H}^n\setminus B_{j+1}}\frac{k_i^{p-1}+k_i^{q-1}}{\|\xi_0^{-1}\circ\xi\|^{Q+sp}_{\mathbb{H}^n}}+
a^+_{j+1}\frac{k_i^{q-1}}{\|\xi_0^{-1}\circ\xi\|^{Q+tq}_{\mathbb{H}^n}}\,d\xi\\
&\leq C2^{i(Q+sp+tq)}\left(\frac{N\omega(r_j)^{p-1}}{r_j^{sp}\sigma^{\beta(p-1)}}+a^+_{j+1}\frac{\omega(r_j)^{q-1}}{r_j^{tq}\sigma^{\beta(q-1)}}
+\frac{k_i^{p-1}+k_i^{q-1}}{r_{j+1}^{sp}}+a^+_{j+1}\frac{k_i^{q-1}}{r_{j+1}^{tq}}\right) \\
&\leq C2^{i(Q+sp+tq)}\left(\frac{Nk_i^{p-1}}{\varepsilon^{p-1}r_j^{sp}\sigma^{\beta(p-1)}}+a^+_{j+1}\frac{k_i^{q-1}}{\varepsilon^{q-1}r_j^{tq}\sigma^{\beta(q-1)}}
+\frac{Nk_i^{p-1}}{r_{j+1}^{sp}}+a^+_{j+1}\frac{k_i^{q-1}}{r_{j+1}^{tq}}\right) \\
&\leq CN2^{i(Q+sp+tq)}\left(\frac{\sigma^{sp-\frac{sp(p-1)}{q-1}}k_i^{p-1}}{r_{j+1}^{sp}}+a^+_{j+1}\frac{\sigma^{tq-sp}k_i^{q-1}}{r_{j+1}^{tq}}
+\frac{\overline{G}(k_i)}{k_i}\right) \\
&\le CN2^{i(Q+sp+tq)}\frac{\overline{G}(k_i)}{k_i}.
\end{align*}
Therefore,
\begin{equation}
\label{3-28}
J_2\le CN2^{i(Q+sp+tq)}\overline{G}(k_i)A_i.
\end{equation}

On the other hand, making use of Lemma \ref{Le107} with $u:=w_i$ yields that
\begin{align}
\label{3-24}
A^\frac{1}{\gamma}_{i+1}\overline{G}(k_i-k_{i+1})
&\le\left(\fint_{B^{i+1}}\left(\left|\frac{w_i}{r^s_{j+1}}\right|^p+a^+_{j+1}\left|\frac{w_i}{r^t_{j+1}}\right|^q\right)^\gamma\,d\xi\right)
^\frac{1}{\gamma} \nonumber\\
&\leq CN\left(\frac{D_1(\hat{\rho}_i,\rho_{i+1})}{r_{j+1}^{sp}}+\frac{\widetilde{D}_1(\hat{\rho}_i,\rho_{i+1})}{r_{j+1}^{tq}}\right)
\fint_{\hat{B}^i}\int_{\hat{B}^i}\frac{H(\xi,\eta,|w_i(\xi)-w_i(\eta)|)}{\|\eta^{-1}\circ\xi\|^Q_{\mathbb{H}^n}}\,d\xi d\eta \nonumber\\
&\quad+CN(1+D_2(\hat{\rho}_i,\rho_{i+1})+\widetilde{D}_2(\hat{\rho}_i,\rho_{i+1}))\fint_{\hat{B}^i}\left|\frac{w_i}{r^s_{j+1}}\right|^p
+a^-_{j+1}\left|\frac{w_i}{r^t_{j+1}}\right|^q\,d\xi.
\end{align}
Thanks to the definitions of $D_1,D_2,\widetilde{D}_1,\widetilde{D}_2$ and $\hat{\rho}_i,\rho_{i+1}$, we from $\hat{\rho}_i\approx\rho_{i+1}\approx r_{j+1}$ and $\hat{\rho}_i-\rho_{i+1}=2^{-i-3}r_{j+1}$ calculate
$$
\frac{D_1(\hat{\rho}_i,\rho_{i+1})}{r_{j+1}^{sp}}\le C2^{i(Q+sp+p)}, \quad \frac{\widetilde{D}_1(\hat{\rho}_i,\rho_{i+1})}{r_{j+1}^{tq}}\le C2^{i(Q+tq+q)}
$$
and
$$
D_2(\hat{\rho}_i,\rho_{i+1})\le C2^{i(Q+sp)}, \quad \widetilde{D}_2(\hat{\rho}_i,\rho_{i+1})\le C2^{i(Q+tq)}.
$$
It is easy to obtain
\begin{equation}\label{3-26}
\fint_{\hat{B}^i}\left|\frac{w_i}{r^s_{j+1}}\right|^p
+a^-_{j+1}\left|\frac{w_i}{r^t_{j+1}}\right|^q\,d\xi\le C\fint_{{B}^i}\overline{G}(w_i)\,d\xi\le C\overline{G}(k_i)A_i.
\end{equation}
It follows from \eqref{3-23}--\eqref{3-26} that
\begin{equation*}
A^\frac{1}{\gamma}_{i+1}\overline{G}(2^{-i-1}\varepsilon\omega(r_j))=A^\frac{1}{\gamma}_{i+1}\overline{G}(k_i-k_{i+1})\le
CN^22^{i2(Q+2q)}\overline{G}(k_i)A_i\le CN^2 2^{i2(Q+2q)}\overline{G}(\varepsilon\omega(r_j))A_i
\end{equation*}
and further
$$
A_{i+1}\le CN^{2\gamma}2^{i2(Q+3q)\gamma}A^\gamma_i,
$$
where $\gamma=\min\left\{\frac{p^*_s}{p},\frac{q^*_t}{q}\right\}>1$ and $C$ depends on $n,p,q,s,t,\alpha,[a]_\alpha,\|a\|_{L^\infty},\beta$.

Now if $A_0$ fulfills
\begin{align}
\label{3-27}
A_0=\frac{|2B_{j+1}\cap\{u_j\leq 2\varepsilon\omega(r_j)\}|}{|2B_{j+1}|}\leq (CN^{2\gamma})^{-\frac{1}{\gamma-1}}2^{-\frac{2\gamma(Q+3q)}{(\gamma-1)^2}}=:\mu,
\end{align}
then by Lemma \ref{Le27} we deduce $A_i\rightarrow0$ as $i\rightarrow\infty$. This means
$$
u_j\ge\varepsilon\omega(r_j) \quad\text{a.e. in} \ B_{j+1},
$$
which together with \eqref{3-17} leads to
$$
\mathop{\rm osc}\limits_{B_{j+1}}u\le (1-\varepsilon)\omega(r_j)=(1-\varepsilon)\sigma^{-\beta}\omega(r_{j+1}).
$$

Finally, choosing $\beta\in\left(0,\frac{sp}{q-1}\right)$ small enough such that
$$
\sigma^\beta\geq 1-\varepsilon=1-\sigma^{\frac{sp}{q-1}-\beta},
$$
then $\mathrm{osc}_{B_{j+1}}u\leq\omega(r_{j+1})$, and $\beta$ depends on $n,p,q,s,t,\alpha,[a]_\alpha,\|a\|_{L^\infty}$ and $\|u\|_{L^\infty(\Omega')}$. Indeed, due to \eqref{3-22}, it yields that
$$
A_0\le \frac{C_{\rm log}N^3}{\log\frac{1}{\sigma}}\leq \mu
$$
by picking $\sigma\le \mathrm{exp}\left(-\frac{C_{\rm log}N^3}{\mu}\right)$ with $\mu$ in \eqref{3-27}. That is, we could select $\sigma=\min\left\{\frac{1}{4},\mathrm{exp}\left(-\frac{C_{\rm log}N^3}{\mu}\right)\right\}$ to ensure the condition \eqref{3-27} does hold true. Now we finish the proof.
\end{proof}

\section*{Acknowledgements}
This work was supported by the National Natural Science Foundation of China (No. 12071098), the National Postdoctoral Program for Innovative Talents of China (No. BX20220381) and the Fundamental Research Funds for the Central Universities (No. 2022FRFK060022). 

\section*{Declarations}
\subsection*{Conflict of interest} The authors declare that there is no conflict of interest. We also declare that this
manuscript has no associated data.

\subsection*{Data Availability} Data sharing is not applicable to this article as no datasets were generated or analysed
during the current study.


\end{document}